\documentclass[]{article}

\usepackage{amsfonts,amsmath,amssymb,amsthm,mathrsfs}
\usepackage[numbers,sort]{natbib}
\usepackage[utf8]{inputenc}
\usepackage[english]{babel}
\usepackage{bbm}
\usepackage{enumitem}
\usepackage{color}
\usepackage{comment}
\usepackage[margin=1.045in]{geometry}
\usepackage{setspace}
\usepackage{scalerel}
\usepackage{color}

\usepackage[colorlinks=true,linkcolor=blue,citecolor=blue,urlcolor=blue,pdfborder={0 0 0}]{hyperref}
\usepackage{cleveref}

\crefname{theorem}{Theorem}{Theorems}
\crefname{thm}{Theorem}{Theorems}
\crefname{lemma}{Lemma}{Lemmas}
\crefname{lem}{Lemma}{Lemmas}
\crefname{remark}{Remark}{Remarks}
\crefname{prop}{Proposition}{Propositions}
\crefname{defn}{Definition}{Definitions}
\crefname{corollary}{Corollary}{Corollaries}
\crefname{conjecture}{Conjecture}{Conjectures}
\crefname{question}{Question}{Questions}
\crefname{chapter}{Chapter}{Chapters}
\crefname{section}{Section}{Sections}
\crefname{figure}{Figure}{Figures}

\newtheorem{theorem}{Theorem}

\newtheorem{proposition}[theorem]{Proposition}
\newtheorem{corollary}[theorem]{Corollary}
\newtheorem{lemma}[theorem]{Lemma}
\newtheorem{conjecture}[theorem]{Conjecture}
\theoremstyle{remark}
\newtheorem{remark}[theorem]{Remark}
\newtheorem{definition}[theorem]{Definition}

\numberwithin{theorem}{section}

\DeclareMathOperator{\Z}{\mathbb{Z}}
\DeclareMathOperator{\R}{\mathbb{R}}

\DeclareMathOperator{\N}{\mathbb{N}}

\newcommand{\bbP}{\mathbb{P}}
\newcommand{\bbE}{\mathbb{E}}
\newcommand{\bbG}{\mathbb{G}}
\newcommand{\bbH}{\mathbb{H}}

\newcommand{\E}[1]{\mathbb{E}\left[#1\right]}

\DeclareMathOperator{\pr}{\mathbb{P}}

\newcommand{\ceil}[1]{\lceil#1\rceil}

\newcommand{\Dag}{\stretchrel*{{\dag}}{X}}

\newcommand{\abs}[1]{\lvert#1\rvert}

\newcommand{\sbt}{\scaleobj{0.8}{\bullet}}

\title{Most transient random walks have infinitely many cut times}

\author{Noah Halberstam\thanks{Department of Pure Mathematics and Mathematical Statistics, University of Cambridge. Email: nh448@cam.ac.uk} \and Tom Hutchcroft\thanks{The Division of Physics, Mathematics and Astronomy, California Institute of Technology. Email: t.hutchcroft@caltech.edu}}

\begin{document}

\maketitle

\begin{abstract}

We prove that if $(X_n)_{n\geq 0}$ is a random walk on a transient graph such that the Green's function decays at least polynomially along the random walk, then $(X_n)_{n\geq 0}$ has infinitely many cut times almost surely. This condition applies in particular to any graph of spectral dimension strictly larger than $2$. In fact, our proof applies to general (possibly nonreversible) Markov chains satisfying a similar decay condition for the Green's function that is sharp for birth-death chains. 
We deduce that a conjecture of Diaconis and Freedman (Ann. Probab. 1980)  holds for the same class of Markov chains, and resolve a conjecture of Benjamini, Gurel-Gurevich, and Schramm (Ann. Probab. 2011) on the existence of infinitely many cut times for random walks of positive speed.

\end{abstract}
\section{Introduction}
Let $(x_n)_{n\geq 0}$ be a sequence taking values in some set $\Omega$. A \textbf{cut time} of $(x_n)_{n\geq 1}$ is a time $n\in\Z_{\geq 0}$ for which the sets $\{x_i:i\leq n\}$ and $\{x_i:i>n\}$ are disjoint.
The study of cut times of random walks was initiated by 
Erd\H{o}s and Taylor in 1960 \cite{MR126299}, who proved lower bounds on the densities of cut times for simple random walks on the integer lattices $\Z^d$ for $d\geq 5$, showing that in this case the \emph{doubly infinite} random walk has a positive density of cut times. The lower dimensional cases $d=3,4$ are more complicated, with the singly infinite random walk having an infinite, density zero set of cut times and the doubly infinite random walk having no cut times almost surely; see \cite{MR1423466,MR1386294,MR1118560,MR1189088,MR947005,MR1062056,MR1035664} for highlights of the literature and \cite{MR2985195} for an overview.
Extending these results beyond the simple random walk on $\Z^d$, James and Peres \cite{MR1687097} and Blanchere \cite{MR1983173} proved that every centered, finite-range random walk on a transient Cayley graph has infinitely many cut times almost surely; see also also the recent work \cite{dur34690} for a more robust analysis. The proofs of these results rely on 
delicate estimates on the gradient of the Green's function that are not available in more general settings, with the works  \cite{MR1687097,MR1983173} also employing a case analysis of the different possible transitive low-dimensional geometries.

Indeed, while transience is of course a necessary condition for a random walk to have infinitely many cut times, the converse implication quickly breaks down once we leave the transitive setting: James, Lyons and Peres \cite{MR2459947} constructed an example of a  birth-death chain that is transient but has finitely many cut times almost surely (see also \cite{MR2644879}), and 
Benjamini, Gurel-Gurevich, and Schramm \cite{MR2789585} showed that the same behaviour is possible for random walks on bounded degree graphs.
On the other hand, Benjamini, Gurel-Gurevich, and Schramm \cite{MR2789585} also prove
that a graph is transient if and only if the \textit{expected} number of cut times of the random walk is infinite, which suggests that most `non-pathological' transient random walks should indeed have infinitely many cut times. It is also known that the set of edges crossed by a random walk always spans a recurrent graph almost surely \cite{MR2308594,MR4059006}, a property that holds trivially when there are infinitely many cut times.

In this paper, we prove a new, very easily satisfied criterion for a transient Markov chain to have infinitely many cut times almost surely, applying in particular to any Markov chain in which the Green's function decays at least polynomially along a trajectory of the chain.
  Our result demonstrates that most transient Markov chains arising in examples will have infinitely many cut times almost surely, and, in particular, provides a simple and unified treatment of the transitive locally finite case.

We now state our main theorem.
Let $M=(\Omega,P)$ be an irreducible Markov chain consisting of countable state space $\Omega$ and transition kernel $P$. For each $x\in \Omega$, we write $\bbP_x$ and $\bbE_x$ for probabilities and expectations taken with respect to the law of the Markov chain trajectory $(X_n)_{n\geq0}$ started at $x$, and write $\bbG(x,y)$ for the Green's function $\bbG(x,y)= \sum_{n\geq0} P^n(x,y)=\bbE_x \sum_{n\geq0} \mathbbm{1}(X_n=y)$.
We say that a sequence of non-negative numbers $(a_n)_{n\geq 0}$ \textbf{decays at least polynomially} as $n\to\infty$ if there exists a constant $c>0$ and an integer $N$ such that $a_n\leq n^{-c}$ for every $n\geq N$.

\begin{theorem} \label{thm:mainPolyDecay}
	Let $M=(\Omega,P)$ be a countable Markov chain and
	let $(X_n)_{n\geq 0}$ be a trajectory of $M$ started at some state $x\in \Omega$. If there exists a decreasing bijection $\Phi:[0,\infty)\rightarrow (0,1]$ such that 
	\begin{equation} \label{eq:inteq}
	\int_0^1 \frac{1}{u (1\vee \log \Phi^{-1}(u))}\ \mathrm{d}u=\infty \qquad \text{ and } \qquad 	\limsup_{n\to\infty}\frac{\mathbb{G}(X_n,X_m)}{\Phi(n)} < \infty \quad \text{ a.s.\ for every $m \geq 0$}
	\end{equation}
	then the trajectory $(X_n)_{n\geq 0}$ has infinitely many cut times almost surely. In particular, the same conclusion holds if $\mathbb{G}(X_n,X_m)$ decays at least polynomially as $n\to\infty$ for each fixed $m\geq 0$ almost surely.
\end{theorem}

We stress that the $\Phi^{-1}(u)$ term appearing in \eqref{eq:inteq} denotes the \emph{inverse} of $\Phi$ rather than its reciprocal.
Note that if the Markov chain is irreducible we can replace the 
decay condition appearing here 
  with the condition that $\limsup_{n\to\infty} \Phi(n)^{-1}\mathbb{G}(X_n,X_0)< \infty$  a.s.

\begin{remark}
Theorem \ref{thm:mainPolyDecay} applies to some decay rates that are slightly slower than polynomial, such as that given by
$\Phi(n) = \exp({-\frac{\log n}{\log\log n}})$.
In \cref{sec:sharp}, we discuss how the results of Cs\'aki, F\"oldes, and R\'ev\'esz~\cite{MR2644879} imply that the integral condition of \cref{thm:mainPolyDecay} is sharp for birth-death chains and hence cannot be improved in general.
\end{remark}

Theorem \ref{thm:mainPolyDecay} easily implies various sufficient conditions for a Markov chain trajectory to have infinitely many cut times almost surely. One particularly simple such condition is as follows.

\begin{corollary}
\label{cor:heatkernel}
Let $M=(\Omega,P)$ be a countable Markov chain and
	let $X=(X_n)_{n\geq 0}$ be a trajectory of $M$ started at some state $x\in \Omega$. If for each $y\in \Omega$ there exist constants $C=C_{xy}<\infty$ and $d=d_{xy}>2$ such that $P^n(x,y) \leq C n^{-d/2}$ for every $n\geq 1$, then $X$ has infinitely many cut times almost surely.
\end{corollary}

Note that if $M$ is irreducible then the hypothesis of this corollary is equivalent to the on-diagonal heat kernel estimate $P^n(x,x)=O(n^{-d/2})$ holding for some $d>2$; for graphs, this is (by definition) equivalent to the \emph{spectral dimension} of the graph being strictly larger than $2$. As such, Corollary \ref{cor:heatkernel} is already sufficient to treat most natural examples of transient graphs arising in examples.

\begin{proof}[Proof of Corollary \ref{cor:heatkernel} given Theorem \ref{thm:mainPolyDecay}]
Fix $x,y\in \Omega$ and suppose that $C<\infty$ and $d>2$ are such that $P^n(x,y)\leq Cn^{-d/2}$ for every $n\geq 1$. We have by the Markov property that
\[
\bbE_x\left[\mathbb{G}(X_n,y)\right] = \bbE_x \#\{\text{visits to $y$ after time $n$}\} = \sum_{m=n}^\infty P^m(x,y) \leq C \sum_{m=n}^\infty m^{-d/2} \leq \frac{2C}{d-2} n^{-(d-2)/2}
\]
for every $n\geq 1$, and hence by Borel-Cantelli that 
\[
\bbG(X_{2^k},y) \leq k^2 2^{-(d-2)k/2} \qquad \text{ for all sufficiently large $k$ almost surely.}
\]
If $\tau$ denotes the first time after time $2^k$ that $X$ hits $y$ then the stopped process $(\mathbb{G}(X_m,y))_{m=2^k}^\tau$ is a non-negative martingale, and it follows by the optional stopping theorem that
\[
\bbP_x\left(\text{there exists } m\geq 2^k \text{ such that } \bbG(X_m,y) \geq k^4 2^{-(d-2)k/2} \mid \bbG(X_{2^n},y) \leq k^2 2^{-(d-2)k/2}\right) \leq \frac{1}{k^2}
\]
for all sufficiently large $k$. Thus, a further application of Borel-Cantelli yields that 
\begin{equation}
\label{eq:greendecayfromheatkernel}
\bbG(X_n,y) \leq (\log_2 n)^4 \left(\frac{n}{2}\right)^{-(d-2)/2}
\end{equation}
for all sufficiently large $n$ almost surely. Since $y$ was arbitrary, the hypotheses of Theorem \ref{thm:mainPolyDecay} are satisfied and $X$ has infinitely many cut times almost surely.
\end{proof}

As mentioned above, earlier results concerning random walks on groups relied on relatively fine control of the Green's function and its gradient, which was used to prove the existence of infinitely many cut times via a second moment argument.
The far weaker and more distributed nature of our decay hypothesis causes this second moment argument to break down. Instead, we compare expectations and conditional expectations of certain special types of cut times as the process $(\bbG(X_n,X_0))_{n\geq 0}$ crosses a small exponential scale $[e^{-k-1},e^{-k}]$. Roughly speaking, this allows us to integrate all of the available information across time, compensating for the looser information. See Section \ref{sec:GC} for details.

\medskip

\textbf{Superdiffusive random walks have infinitely many cut times. }
As an application of \cref{thm:mainPolyDecay}, we also prove that walks on graphs and networks (i.e.\ reversible Markov chains) satisfying a weak superdiffusivity condition have infinitely many cut times almost surely. Given a network $N=(V,E,c)$ with underlying graph $(V,E)$ and conductances $c:E\rightarrow (0,\infty)$, we define the conductance $c(v)$ of a vertex $v$ to be the total conductance of all oriented edges emanating from $v$.
\begin{theorem} \label{thm:speedTheorem}
	Let $N=(V,E,c)$ be a locally finite, connected network with $\inf_v c(v)>0$ and let $X$ be a random walk on $N$.
	 If there exists $r>3/2$ such that
	 \begin{equation}
     \label{eq:superdiffusivity}
\liminf_{n\to\infty} \frac{d(X_0,X_n)}{n^{1/2}(\log n)^r}>0 \qquad \text{ almost surely},
\end{equation}
then $X$ has infinitely many cut times almost surely.
\end{theorem}

This result resolves a conjecture of Benjamini, Gurel-Gurevich, and Schramm \cite{MR2789585}, who asked whether random walks on graphs with positive linear $\liminf$ speed have infinitely many cut times almost surely. For \emph{bounded degree} graphs where the walk has positive speed, our proof yields that the walk has a \emph{positive density} of cut times a.s., yielding a very strong version of their conjecture.

\begin{theorem} \label{thm:speedTheorem_density}
	Let $G$ be a bounded degree graph and let $X$ be a random walk on $G$.
	 \begin{equation*}
\text{If} \quad \liminf_{n\to\infty} \frac{1}{n}d(X_0,X_n)>0 \quad \text{a.s.} \quad \text{then} \quad \liminf_{n\to\infty}\frac{1}{n}\#\{0\leq m \leq n: m \text{ is a cut time for $X$}\} >0 \quad \text{a.s.}
\end{equation*}
\end{theorem}


Note that Theorem \ref{thm:speedTheorem} is not an immediate consequence of Theorem \ref{thm:mainPolyDecay}, as we are not aware of any general result allowing us to deduce Green's function decay estimates from distance estimates without further assumptions on the graph: the Varopoulos-Carne inequality \cite{MR822826,MR837740} tells us that $p_m(X_n,X_0)$ is small when $d(X_0,X_n)$ is much larger than $m^{1/2}$, but does not give any control whatsoever of the large-time contribution to the Green's function $\sum_{m \geq n^2} p_m(X_n,X_0)$.
 To circumvent this obstacle, we consider adding a spatially-dependent killing to our network. We tune the rate of killing to be weak enough that the walk has a positive chance to live forever when superdiffusive, and strong enough that we can control the decay of the killed Green's function along the walk. We prove that this killed walk has infinitely many cut times almost surely on the event that it survives forever, from which Theorem \ref{thm:speedTheorem} easily follows.

\medskip
\textbf{The Diaconis-Freedman conjecture.} Let $M=(\Omega,P)$ be a transient Markov chain, and let $X=(X_n)_{n\geq 0}$ be a trajectory of $M$.  
The \textbf{partially exchangeable $\sigma$-algebra} of $X$ is defined to be the exchangeable $\sigma$-algebra generated by the sequence of increments $((X_n,X_{n+1}))_{n\geq 0}$, that is, the set of events that are determined by the sequence of increments and that are invariant under permutations of this sequence that fix all but finitely many terms. This $\sigma$-algebra arises naturally in the work of Diaconis and Freedman \cite{MR556418}, who proved that every partially exchangeable sequence of random variables can be expressed as a Markov process in a random environment. This can be thought of as a partially-exchangeable version of de Finetti's theorem and plays an important role in the theory of reinforced random walks \cite{MR3189433,MR2441859}. Their study of the partially exchangeable $\sigma$-algebra led Diaconis and Freedman to make the following conjecture. Given a trajectory $X=(X_n)_{n\geq 0}$, we define  the \textbf{crossing number} of an ordered pair of states $(x,y)$ to be the number of integers $n$ such that $(X_n,X_{n+1})=(x,y)$.

\begin{conjecture}[Diaconis-Freedman 1980]
Let $X$ be a trajectory of a transient Markov chain. Then the partially exchangeable $\sigma$-algebra of $X$ is generated by the crossing numbers of $X$.
\end{conjecture}

An equivalent statement of this conjecture is that if we condition on the crossing numbers then the resulting process has trivial exchangeable $\sigma$-algebra almost surely.
Note that there is a close analogy between this conjecture and the problem of computing the Poisson boundary for lamplighter random walks~\cite{MR4228277,MR704539}.
As observed in \cite{MR1687097}, it is easily seen that the Diaconis-Freedman conjecture holds whenever $X$ has infinitely many cut times almost surely. 
As such, our main results imply that the Diaconis-Freedman conjecture holds for most transient Markov chains arising in examples.

\begin{corollary}
		Let $M=(\Omega,P)$ be an irreducible transient Markov chain with trajectory $(X_n)_{n\geq 0}$. If $(M,x)$ satisfies the hypotheses of either Theorem \ref{thm:mainPolyDecay} or Corollary \ref{cor:heatkernel} then the partially exchangeable $\sigma$-algebra of $X$ is generated by its crossing numbers.
\end{corollary}

\medskip
\noindent \textbf{Organisation.} Section \ref{sec:GC} contains the proof of our main theorem, Theorem \ref{thm:mainPolyDecay}. First, in Section \ref{subsec:overarching}, we describe the overarching strategy behind the proof of Theorem \ref{thm:mainPolyDecay} and give a proof in the much simpler special case in which the Green's function decays exponentially along the random walk. We then introduce relevant technical preliminaries in Sections \ref{subsec:drawbridge} and \ref{subsec:deterministic} before proving Theorem \ref{thm:mainPolyDecay} in Section \ref{subsec:mainproof}.  
Finally, we prove our results concerning superdiffusive walks in Section \ref{sec:SSD} and prove that \cref{thm:mainPolyDecay} is sharp for birth-death chains in \cref{sec:sharp}.

\medskip

\noindent \textbf{Notation.} Given a sequence of real numbers $(z_n)_{n\geq 0}$, we will often write $(z^*_n)_{n\geq 0}$ for the associated sequence of running minima $z_n^*= \min_{0 \leq m \leq n} z_m$.

\section{Proof of the main theorem} \label{sec:GC}

In this section we prove our main theorem, Theorem \ref{thm:mainPolyDecay}. We will work mostly under the additional assumption that $M$ is irreducible, locally finite (i.e.\ that there are finitely many possible transitions from each state), and has $P(x,x)=0$ except possibly for one absorbing state $\dagger$, before showing that the general case follows from this case at the end of the proof. It will be convenient to work throughout with the hitting probabilities
\[
\bbH(x,y) = \bbP_x(\text{hit $y$}) = \frac{\bbG(x,y)}{\bbG(y,y)}
\]
rather than the Green's function. This can be done with minimal changes to each of the other statements since $\bbH(X_n,y)$ decays at the same rate as $\bbG(X_n,y)$ for each fixed $y$.

Let us now give some relevant definitions. We define a \textbf{Markov chain with killing} to be a tuple $M=(\Omega,P,\dagger)$ where $\Omega$ is a countable state space, $P:\Omega\times\Omega\rightarrow[0,1]$ is the transition kernel
 and $\dagger\in\Omega$ is a distinguished \textbf{graveyard} state satisfying $p(\dagger,\dagger)=1$. 
We say that a Markov chain with killing is \textbf{locally finite} if the set $\{v:p(u,v)>0\}$ is finite for every $u\in \Omega$ and say that a Markov chain with killing is \textbf{irreducible} if for every $u,v\in\Omega\setminus\{\dagger\}$ there exists $n\in\N$ such that $P^{n}(u,v)>0$. We say the chain is \textbf{transient} if every state other than $\dagger$ is visited at most finitely many times almost surely.
Given a trajectory $X$ of a Markov chain with killing, we define for each $x\in \Omega$ the hitting time $\tau_x=\inf\{n\geq 0: X_n=x\}$,
 and say that a trajectory of the chain is \emph{killed} if $\tau_{\Dag}<\infty$. 

\begin{theorem} \label{thm:kcuts}
Let $M=(\Omega,P,\dagger)$ be a transient, locally finite, irreducible Markov chain with killing such that $P(x,x)=0$ for every $x\neq \dagger$, let $X=(X_n)_{n\geq 0}$ be a trajectory of  $M$, and let $\phi:[0,\infty)\rightarrow[0,\infty)$ be an increasing bijection such that
	\begin{equation}\label{eq:general_condition}
	\sum_{n=1}^\infty \frac{1}{1 \vee \log\left(\phi^{-1}(n)\right)}=\infty.
	\end{equation} 
	If the event
	$\mathscr{G}=\{ \limsup_{n\to\infty} e^{\phi(n)} \bbH(X_n,X_0) < \infty\}$
	has positive probability, then $X$ is either killed or has infinitely many cut times almost surely conditional on $\mathscr{G}$.
\end{theorem}

Note that \eqref{eq:general_condition} becomes equivalent to \eqref{eq:inteq} when $\Phi(x)=e^{-\phi(x)}$ as established in the following lemma; we found the condition in terms of $\Phi$ given in Theorem \ref{thm:mainPolyDecay} to be easier to think about in examples, while the condition in terms of $\phi$ given in Theorem \ref{thm:kcuts} is better suited to the proof.

\begin{lemma}
\label{lem:integral}
Let $\Phi :[0,\infty)\to(0,1]$ be a decreasing bijection and let $\phi=-\log \Phi$. Then
\[
\int_0^1\frac{1}{u (1\vee\log \Phi^{-1}(u))} \mathrm{d} u = \infty \qquad \text{ if and only if } \qquad \sum_{n=1}^\infty \frac{1}{(1 \vee \log \phi^{-1}(n))} = \infty.
\]
\end{lemma}

\begin{proof}[Proof of Lemma \ref{lem:integral}]
We will prove that if the integral involving $\Phi$ diverges then the sum involving $\phi$ diverges, this being the only direction of the lemma that we need. The reverse direction is proved similarly. Since $\Phi$ is decreasing, we have that
\begin{multline}
\int_0^1\frac{1}{u (1 \vee \log \Phi^{-1}(u))} \mathrm{d} u = \sum_{k=1}^\infty \int_{e^{-k}}^{e^{-k+1}} \frac{1}{u \log (1 \vee \Phi^{-1}(u))} \mathrm{d} u
\\ \leq \sum_{k=1}^\infty  \frac{e^{-k+1}}{e^{-k}(1 \vee \log \Phi^{-1}(e^{-k+1}))} = \sum_{k=1}^\infty  \frac{e}{(1 \vee \log \Phi^{-1}(e^{-k+1}))},
\end{multline}
and the claim follows since $\Phi^{-1}(e^{-k+1}) =\phi^{-1}(k-1)$.
\end{proof}

\subsection{The overarching strategy and the special case of exponential decay}
\label{subsec:overarching}
In this section we describe the high-level strategy underlying Theorem \ref{thm:mainPolyDecay} and present a proof  in the much simpler case of an exponentially decaying hitting probability process $\mathbb{H}(X_n)$. We then document the issues that arise when attempting to extend this method to the subexponential case and outline how we overcome them.

The high-level idea is to construct a function $F:\Omega\rightarrow[0,\infty)$ such that there are infinitely many times $n$ when the trajectory $(X_m)$ of the irreducible Markov chain $M=(\Omega,P)$ satisfies
\begin{center}
	1.  $F(X_{n})<\min_{m<n} F(X_{m})$. \qquad\qquad\qquad and \qquad\qquad\qquad
	2. $F(X_m)\leq F(X_{n})$ for $m>n.$
\end{center}
In other words, at each of these times $n$, the process $F(X_i)$ must \textit{drop} lower than it has previously, and this drop must be a \textit{permadrop}, i.e.\ $F(X_i)$ must not \textit{recover} to any level achieved prior to the drop. Indeed, if both 1.\ and 2.\ hold then the walk cannot return to any vertex it has previously visited and therefore has a cut time at $n$.
For the first of these two properties to hold infinitely often, it is sufficient that the process $(F(X_n))_{n\geq 0}$ converges to zero, and given transience of the Markov chain, a candidate such as $F(x) = d(o,x)^{-1}$ would suffice. Indeed, studying graph distances appears to be a particularly natural choice in the superdiffusive regime.
Unfortunately, there seem to be very limited tools available to prove that this function yields infinitely many permadrops, even when the random walk has positive speed.

These considerations make it natural to instead study the decay of hitting probabilities along the random walk:
when the chain is transient 
 the \textbf{hitting probability process} $(\mathbb{H}(X_n,X_0))_{n\geq 0}$ automatically tends to zero, and we can use the fact that the process is a martingale to attempt to analyse the number of permadrops.
   Indeed, Benjamini, Gurel-Gurevich and Schramm  \cite{MR2789585} used martingale techniques to show that the \emph{expected} number of permadrops of this process is always infinite when the chain is transient and hence that every transient chain has infinitely many cut times in expectation. Thus, a natural approach to the cut times problem is to find sufficient conditions for the number of permadrops of this process to be infinite almost surely.

 Let us first consider the special case in which $M$ is irreducible and $\mathbb{H}(X_n,X_0)$ decays exponentially. Note that this case is already sufficient to resolve the conjecture of Benjamini, Gurel-Gurevich, and Schramm~\cite{MR2789585} in conjunction with the spatially-dependent-killing argument of Section \ref{sec:SSD}.

 \begin{proposition} \label{prop:expDeccase}
If $M$ is irreducible and the hitting probability process $Z_n := \mathbb{H}(X_n,X_0)$ decays exponentially in the sense that $\liminf_{n\to\infty} \frac{1}{n} \log 1/Z_n >0$  a.s.\ then  $X$ has infinitely many cut times a.s. 
 \end{proposition}

  The proof of this proposition will rely on L\'{e}vy's zero-one law \cite{LevyPaul1954Tdld}, which is a special case of the martingale convergence theorem.

 \begin{lemma}[L\'{e}vy's zero-one law]
 	Let $(\Omega,F,\pr)$ be a probability space, and let $\mathbb{E}$ denote expectation with respect to $\pr$. Let $(F_n)_{n\geq 0}$ be a filtration and let $A$ be an $F_\infty=\cup_n F_n$ measurable event. Then
 	\[
 	\lim_{k\rightarrow\infty}\mathbb{P}[A\mid F_k] = \mathbbm{1}_A \qquad \text{almost surely.}
 	\]
\end{lemma}

 \begin{proof}[Proof of Proposition \ref{prop:expDeccase}]
 	Since $M$ is irreducible, $Z_n$ is positive for every $n\geq 0$. Since $Z$ also decays exponentially almost surely, the sequence $Z^*_n=\inf_{m\leq n}Z_m$ is also positive and decays exponentially almost surely. In particular, there exists a $[0,1]$-valued random variable $\alpha$ satisfying $\alpha<1$ almost surely such that 
 	\begin{equation} \label{eq:jump}
 	Z_n = Z_n^* \leq \alpha Z^*_{n-1}
 	\end{equation}
 	for infinitely many $n$.

 	For each $a\in(0,1)$, define the sequence of times $(t_n^a)_{n\geq 0}$ recursively by setting $t_0^a=0$ and for each $n\geq 1$ setting
 	\[
 	t_n^a = \inf \{m>t_{n-1}^a:Z_m \leq a Z^*_{m-1}\}
 	\]
 	with the convention that $\inf \emptyset = \infty$.
Let $E^a=\{t_n^a<\infty\ \forall n\geq 0\}$ be the event that there are infinitely many drops of size at least $a$ and for each $n \geq 1$ consider the \emph{permadrop} event
 	 $A_n^a=\{t_n^a < \infty $ and $Z_m<a^{-1}Z_{t_{n}^a}$ for every $m \geq t^a_n\}$, so that $Z_m < Z^*_{t_n^a-1}$ for every $m\geq t_n^a$ on the event $A^a_n$.
 	Let $\mathcal{A}^a$ be the event that infinitely many of the events $A^a_n$ hold, so that $\mathcal{A}^a \subseteq E^a$ and $X$ has infinitely many cut times whenever $\mathcal{A}^a$ holds. 
 	Fix $a\in (0,1)$ and suppose that $E^a$ occurs with positive probability.
 	 Since the filtration
 	has the $\sigma$-algebra generated by the entire random walk as its union,
 	  L\'{e}vy's zero-one law implies that
 	\[
 	\lim_{n\rightarrow\infty}\pr(t_n^a < \infty \text{ and } \exists m\geq n \text{ s.t. } A_m^a \text{ occurs} \mid \mathcal{F}_{t_n})=\mathbbm{1}(\mathcal{A}^a)
 	\] almost surely.
 	On the other hand, since $Z$ is a supermartingale,
 	we have by optional stopping that
 	\[
 	\pr(t_n^a < \infty \text{ and } \exists m\geq n \text{ s.t. } A_m^a \text{ occurs} \mid \mathcal{F}_{t_n})\geq \pr(A_n^a\mid \mathcal{F}_{t_n})\mathbbm{1}(t_n^a<\infty)\geq (1-a)\mathbbm{1}(t_n^a<\infty)
 	\]
almost surely for each $n\geq 1$. Since the latter estimate is bounded away from zero as $n\to\infty$ on the event $E^a$, we deduce that $\mathcal{A}^a$ holds almost surely conditional on $E^a$ and hence that $X$ has infinitely many cut times almost surely conditional on $E^a$. The claim follows since $a \in (0,1)$ was arbitrary and the countable union $\bigcup_{k \geq 1} E^{(k-1)/k}$ has probability $1$ by \eqref{eq:jump}. \qedhere
 \end{proof}

\textbf{Problems in the subexponential case.} As we have just seen, it is straightforward to show that $Z_n$ has infinitely many permadrops whenever it decays exponentially: the large decay rate guarantees an infinite supply of drops of a constant relative size, and the optional stopping theorem bounds the probability of each of these drops being a permadrop below by a constant. This constant lower bound means we can rely on soft techniques like L\'{e}vy's zero-one law to deduce that permadrop events occur infinitely often without having to worry about their dependencies. However, even if we did have to think about dependencies, we could choose the drops far enough away from each other such that we could easily control the correlations between their recovery events. (The exact argument is somewhat subtle: it is not necessarily true that the correlations are small, but the conditional probability of there being a permadrop on one scale given what has happened on previous scales is bounded away from 0.)

When we move to the subexponential case, this argument quickly begins to break down. Indeed, the best we were able to do by optimizing the above approach was to handle the case of stretched-exponential decay $Z_n = e^{-\Theta(n^z)}$ for $z>1/2$. Let us now overview the problems that arise when attempting to perform such an optimization. First, without access to L\'{e}vy's zero-one law, we now have to consider correlations between recovery events. Perhaps more significantly, however, subexponential decay gives us only very loose information about the local behaviour of the hitting probability process. We know the extent to which it must decrease over long periods of time, but have relatively little structural information about how this decrease occurs or about the positions and sizes of the drops: the overall fall in value of the process could be made up of frequent small drops, rare large drops, or any combination thereof.
Consider for instance the case of stretched exponential decay $Z_n =  e^{-\Theta(n^{z})}$ for $z\in(0,1)$. This decay could be achieved by drops by a factor of size $1-n^{z-1}$ at a positive density of times, or, say, by halving at each time of the form $n^{1/z}$. The only restriction is that we cannot have too much of the decay made up of very small drops, as this would contradict the assumed decay of the process.

In an attempt to adapt the arguments used in the exponential decay case, a natural starting place would be to attempt to extract a sparse sequence of roughly independent drops of guaranteed size. For instance, in the stretched exponential decay case $Z_n= e^{-\Theta(n^{z})}$, we can set up the infinite sequence of stopping times \[
t_n=\inf\left\{m>t_{n-1}:Z_m<a_nZ_{t_{n-1}} \text{ and } Z_m<(1-n^{z^\prime-1}) Z_m^*\right\}
\]
for some decreasing sequence $(a_n)$ very slowly converging to $0$, and $z^\prime\in(0,z)$, where the sequence $a_n$ should be chosen to allow us to safely ignore dependencies between successive steps. It turns out that this works well for $z>1/2$: a deterministic argument proves that if $(Z_m)$ has only finitely many drops of any constant relative size, then for $n$ large enough, the drop at time $t_n$ must approximately have size at least $1-n^{(z^\prime-1)/z}$, and optional stopping allows us to control the dependencies between the recovery events. Optional stopping then gives a $n^{(z^\prime-1)/z}$ probability of the drop at time $t_n$ being a permadrop, and a simple generalisation of Borel-Cantelli then implies that there are infinitely many permadrops almost surely. For $z\leq\frac{1}{2}$, however, the sequence $n^{(z^\prime-1)/z}$ has a convergent sum and the argument breaks down. At this stage we are very far from handling polynomial decay!

\medskip

\textbf{Addressing the problems. }
To get results when the process decays slower than $e^{-n^{1/2}}$, we can no longer just extract sparse sequences and must begin to consider neighbouring drops and the interactions between their recovery events. We attempted to employ a second moment method, bounding each $\pr(A_i\cap A_j)$ from above where $A_i$ is the probability that the $i$th drop is a permadrop. Unfortunately, due to the looseness of the information that we have regarding the locations and sizes of the drops, this method proved difficult to implement and did not seem capable of producing optimal results.
To overcome the outlined issues, we instead analyse the path of the hitting probability process as it traverses a series of spatial scales. At each scale we upper bound the expected number of \emph{large} permadrops \textit{conditional on there being at least one}, and simultaneously \textit{lower bound} the \emph{unconditional} expected number of large permadrops. We modulate the definition of ``large" across scales to ensure that the former  quantity is not too large and the latter is not too small: we need the threshold for the drop sizes we consider to be small enough that we get an adequate supply of drops to lower bound the unconditional expectation while being large enough to prevent an accumulation of drops amplifying the conditional expectation. Once we have done this with a well-chosen choice of thresholding function, comparing these two quantities allows us to lower bound the probability that there is a permadrop on each scale; 
considering a whole scale simultaneously, rather than individual pairs of drops, allowed us to tackle the flexibility present in the structure of the decay. The Borel-Cantelli counterpart then has a natural application demonstrating that  there are infinitely many permadrops when the decay of the hitting probability is strong enough.

Rather than working directly with the hitting probability process of the Markov chain, we work with an augmented continuous time process which we call the \textit{drawbridge process}. This makes the hitting probability process a continuous martingale away from $1$ and lets us use optional stopping to get exact expressions for permadrop probabilities rather than one-sided inequalities: this is important since we need to prove both upper and lower bounds on relevant expectations. As mentioned above, we will work primarily in the setting of locally finite Markov chains that are irreducible bar the presence of a graveyard state, before deducing a result for general Markov chains via a simple reduction argument.

\subsection{The drawbridge process}
\label{subsec:drawbridge}

 At several points in our analysis we will want to apply the optional stopping theorem to get equalities rather than one-sided inequalities,
making it convenient to work with continuous rather than discrete martingales.
 For the random walk on a graph, it is well-known that one can embed the discrete-time random walk inside a continuous-time continuous process by considering Brownian motion on an appropriately constructed metric graph known as the \textit{cable graph} \cite{MR3152724,MR3502602}. We now construct a similar way of embedding a \emph{non-reversible} locally finite Markov chain inside a continuous Markov process, which we call the \emph{drawbridge process}, and hence of embedding the discrete-time hitting probability process inside a continuous martingale.
  While there are precedents for considering similar processes \cite{MR2009371}, it appears to be much less well known than the cable process, and we give a fairly detailed introduction to keep the paper self-contained.

Before giving a precise definition let us first give the intuition behind the name. Let $M=(\Omega,P,\dagger)$ be a locally finite Markov chain with killing and suppose that $P(x,x)=0$ for every $x\neq \dagger$. Consider the corresponding directed graph $G$ with vertex set $\Omega$ and with a directed edge from a vertex $u$ to a vertex $v$ if $u\neq v$ and $P(u,v)>0$. We can make this abstract graph physical, in some sense, by assigning the positive real length $1/P(u,v)$ to each directed edge $(u,v)$. 
 While it is nonsensical to think of a Brownian motion which can only travel in one direction, we can recover restrictions in motion through the use of ``drawbridges". More specifically, we envision Brownian motion on a modified version of the metric graph, in which one places a ``drawbridge" along each directed edge of the metric graph. Each drawbridge has two states, \textbf{raised} and \textbf{lowered}. When the drawbridge at $(u,v)$ is raised, the connection between the part of the edge near $v$ and the vertex $v$ itself is severed, and the Brownian motion cannot cross from $v$ onto the edge $(u,v)$. Conversely, when $(u,v)$ is lowered it is possible for the Brownian motion to enter the edge from either $u$ or $v$. For each vertex $u$, we call the drawbridges across the edges emanating from $u$ in the corresponding directed graph the \textbf{outgoing drawbridges} from $u$. The drawbridge process will be defined by taking the Brownian motion on this metric graph and raising and lowering drawbridges as the Brownian motion moves so that, at each time, the outgoing drawbridges from the last vertex it visited are lowered and all other drawbridges are raised.

We now make this precise.
Let $M=(\Omega,P,\dagger)$ be a locally finite Markov chain with killing such that $P(x,x)=0$ for every $x\neq \dagger$. For each state $x\in\Omega$, we define the set of outgoing states $x^\rightarrow$ to be $\{y \in \Omega\setminus \{x\}:P(x,y)>0\}$. We define the \emph{star graph} $S[x]$ to be the metric graph with vertex set $\{x\}\cup x^\rightarrow$, with edge set $\{\{x,y\}:y\in x^\rightarrow\}$, and with edge lengths 
$1/P(x,y)$, so that $S[\dagger]$ is the metric graph consisting of the single vertex $\{\dagger\}$ and no edges. In an abuse of notation, we will identify vertices in the star graph with their corresponding states; the precise meaning will be clear from context.
We construct the metric space $\mathcal{S}$ from the disjoint union $\mathcal{S}^\sqcup=\sqcup_{y\in\Omega}S[y]=\{(x,y):x\in\Omega, y\in S[x]\}$ by gluing together $(u,u)$ and $(v,u)$ for every $v\in\Omega$ and $u\in v^\rightarrow$. Note that every point in $\mathcal{S}$ has a unique representation of the form $(x,y)$ where $x\in \Omega$ and $y \in S[x]\setminus x^\rightarrow$.

Let $(x_0,y) \in \mathcal{S}$ be such that $x_0\in \Omega$ and $y\in S[x] \setminus x_0^\rightarrow$.
We construct the drawbridge process on $\mathcal{S}$ starting at $(x,y)$ as follows. First we start a Brownian motion $(B^0_t)_{t\geq 0}$ on $\mathcal{S}^\sqcup$ starting at $(x_0,y)$ at time $\mathcal{T}(0)=0$ and run until the stopping time $\mathcal{T}(1)=\inf\{t> \mathcal{T}(0): B^0_t\in \{x_0\}\times x_0^\rightarrow\}$, so that if $\mathcal{T}(1)<\infty$ then $B^0_{\mathcal{T}(1)}=(x_0,x_1)$ for some $x_1\in x_0^\rightarrow$. If $\mathcal{T}(1)$ is finite, we then run a Brownian motion $(B^1_t)_{t=\mathcal{T}(1)}^{\mathcal{T}(2)}$ on $\mathcal{S}[x_1]$, started from $(x_1,x_1)$ at time $\mathcal{T}(1)$ and run until the stopping time $\mathcal{T}(2)=\inf\{t> \mathcal{T}(1): B^1_t\in \{x_1\}\times x_1^\rightarrow\}$.
  We iterate this construction to generate a possibly infinite sequence $((B^i_t)_{\mathcal{T}(i)\leq t\leq \mathcal{T}(i+1)})_i$, noting that $\mathcal{T}(i)$ is almost surely finite whenever $x_i \neq \dagger$. If the sequence terminates because $\mathcal{T}(i)=\infty$ for some $i\in N$, which almost surely happens exactly when the process first visits the graveyard state $\dagger$, then we define $\mathcal{T}(j)=\infty$ for $j>i$, set $\tau_{\Dag}=i-1$ and $\mathcal{T}_\dagger=\mathcal{T}(\tau_\dagger)$. If the sequence does not terminate, we set $\tau_{\Dag}=\mathcal{T}_\dagger=\infty$. Finally, we construct the \emph{drawbridge process} $(\mathcal{X}_t)_{t\geq 0}$ by concatenating, in order, the images of the paths of the  Brownian motions $(B^i)_{0\leq i\leq\tau_{\Dag}
}$ in $\mathcal{S}$ under the gluing map $\mathcal{S}^\sqcup\rightarrow\mathcal{S}$.

We let $\mathbf{P}_{x,y}$, $\mathbf{E}_{x,y}$ denote probability and expectation with respect to the law of $\mathcal{X}$ started at $(x,y)$ and write $\mathbf{P}_x=\mathbf{P}_{x,x}$ and $\mathbf{E}_x=\mathbf{E}_{x,x}$. Observe that if $\mathcal{X}$ is started at $(x,x)$ for some $x\in \Omega$ then the discrete-time process  $X=(X_n)_{n=0}^{\tau_\dagger}$ defined by $(\mathcal{X}_{\mathcal{T}(n)})_{n= 0}^{\tau_\dagger}=(X_n,X_n)_{n= 0}^{\tau_\dagger}$ has the distribution of a trajectory of the Markov chain stopped when it hits the graveyard state $\dagger$.
Thus, if we fix an arbitrary `origin' state $o \neq \dagger$ and let $\mathcal{T}_o=\inf\{t\geq 0:\mathcal{X}_t=(o,o)\}$ be the hitting time of $o$, setting $\mathcal{T}_o=\infty$ if $o$ is never hit, then the hitting probability
$\mathcal{H}(x,y)=\mathbf{P}_{x,y}(\mathcal{T}_o<\infty)$ satisfies $\mathcal{H}(x,x)=\mathbb{H}(x)=\mathbb{H}(x,o)$ for all $x\in\Omega$.
We define $(\mathcal{Z}_t)_{t\geq 0}=(\mathcal{H}(\mathcal{X}_t))_{t\geq 0}$ and $(Z_n)_{n\geq 0}=(\mathbb{H}(X_n))_{n\geq 0}=(\mathcal{Z}_{\mathcal{T}(n)})_{n\geq 0}$, noting that if $\sigma_1 \leq \sigma_2$ are stopping times for $\mathcal{X}$ such that $\mathcal{X}_t \neq o$ almost surely for every $\sigma_1 < t < \sigma_2$ 
  then $(\mathcal{Z}_t)_{t=\sigma_1}^{\sigma_2}$ is a continuous time martingale with respect to its natural filtration.

\subsection{Deterministic preliminaries}
\label{subsec:deterministic}

We now set up the notation to record the behaviour of the running minima of a non-negative sequence as it converges to zero across a series of exponential scales. Recall that for each sequence $(z_n)_{n\geq 1}$ we write $(z_n^*)_{n\geq 0} = (\min_{m\leq n} z_m)_{n\geq0}$ for the associated sequence of running minima.  Given a non-negative sequence $(z_n)_{n\geq 0}$ converging to $0$ with $z_0>0$, we construct a sequence of logarithmic scales over which to analyse and control its behaviour. We record the associated notation in the following definition.

\begin{definition}[Notation for drops at scale $k$]\label{def:drops}
Fix a non-negative sequence $z=(z_n)_{n\geq 0}$ with $z_n>0$ and $z_n \to 0$ as $n\to\infty$.
Let $k_0=k_0(z)=\ceil{2\log z_0^{-1}}$, and for each $k\geq 1$ define the $k$th scale interval
  $I_n=[e^{-{k-1}},e^{-k}]$. For each $k\geq k_0$ we define the set $D_k$ by adjoining the set of running minima on the $k$th scale to the endpoints of the corresponding interval $I_k$ so that 
  \[D_k=D_k(z)=(\{z_m^* : m \geq 1\}\cap I_k)\cup\{ e^{-k}, e^{-k-1}\}\] for each $k\geq k_0$.  We define $N_k=N_k(z)=\abs{D_k(z)}-1$ and label the elements of $D_k(z)$ in decreasing order as $(d_{i,k})_{1\leq i\leq N_k}=(d_{i,k}(z))_{0\leq i\leq N_k}$ so that
\[ e^{-k}= d_{0,k}>d_{1,k}>\cdots>d_{N_k,k}= e^{-k-1} \qquad \text{ and } \qquad \prod_{i=0}^{N_k-1} \frac{d_{i+1,k}}{d_i,k} = e^{-1}.\]
We call the pairs $(d_{i,k},d_{i+1,k})$ for $0\leq i \leq N_k-1$ the \textbf{drops} of $z$ on scale $k$.
	When context makes clear which sequence $z$ we are referring to, we will drop it from our notation. Similarly, when it is clear that we are referring to a particular scale $k$ we will drop the second subscript on the $d_{i,k}$ by writing $d_i=d_{i,k}$.
\end{definition}

We begin the proof by proving the following deterministic lemma. Roughly speaking, this lemma says that for sequences which decay to zero sufficiently quickly, we can define a threshold between large and small drops in such a way that the following hold:
\begin{enumerate}
\item For a good proportion of scales, a good proportion of the decay is made up of large drops.
\item The large drops are large enough that there cannot be too many such drops at any particular scale.
\end{enumerate} The actual threshold we will use will be a simple function of the $\psi$ which is outputted by this lemma. Later, we will apply this lemma to the hitting probability process. 
Given an increasing function $\psi:[0,\infty)\to [0,\infty)$ and a sequence of non-negative numbers $(z_n)_{n\geq 0}$, we say that $(z_n)_{n\geq 0}$ is \textbf{$\psi$-good on scale $k$} if
\begin{multline*}
\prod\left\{\frac{d_{i+1,k}}{d_{i,k}}:0\leq i < N_{k} \text{ is such that } \frac{d_{i+1,k}}{d_{i,k}} \leq \exp\left[-\frac{k}{2\psi^{-1}(k)}\right]\right\} \\\leq \prod\left\{\frac{d_{i+1,k}}{d_{i,k}}:0\leq i< N_{k} \text{ is such that }  \frac{d_{i+1,k}}{d_{i,k}} > \exp\left[-\frac{k}{2\psi^{-1}(k)}\right]\right\},
\end{multline*}
or equivalently if
\[
\prod\left\{\frac{d_{i+1,k}}{d_{i,k}}:0\leq i< N_{k} \text{ is such that } \frac{d_{i+1,k}}{d_{i,k}} \leq \exp\left[-\frac{k}{2\psi^{-1}(k)}\right]\right\} \leq e^{-1/2},
\]
where we write $\prod\{A_i:i\in I\}=\prod_{i\in I} A_i$ for reasons of legibility.  That is, $(z_n)_{n\geq 0}$ is $\psi$-good on scale $k$ if at least half of the total decay across the scale comes from drops of size at least $\Psi(k):=e^{-k/2\psi^{-1}(k)}$ in a geometric sense. Note that $\psi^{-1}$ denotes the inverse of $\psi$ defined by $\psi^{-1}(x)=\min\{ y \geq 0 : \psi(y) \geq x\}$; we will typically think of $\psi$ as being a slowly growing function so that $\psi^{-1}$ satisfies $\psi^{-1}(x)\gg x$.

\begin{lemma}[Good, well-separated scales] \label{lem:decaydensity}
	Let $\phi:[0,\infty)\rightarrow [0,\infty)$ be an increasing bijection 
	 such that
	\begin{equation} \label{eq:sumeq}
		\sum_{k=1}^\infty \frac{1}{1 \vee \log\left(\phi^{-1}(k)\right)}=\infty.
	\end{equation}
	Then there exist an increasing bijection $\psi:[0,\infty)\rightarrow [0,\infty)$ satisfying $\psi(x)\leq \phi(x)$ and  $\psi(x)\leq \sqrt{x}$ for every $x\geq 0$ and a strictly increasing function $a:\N\rightarrow\N$ satisfying
	\begin{equation}\label{eq:freq}
		\lim_{n\rightarrow\infty} a(k)-a(k-1)=+\infty
		\qquad \text{ and } \qquad \sum_{k=1}^\infty \frac{1}{1 \vee \log(\psi^{-1}(a(k)))}=\infty,
	\end{equation}
	 such that if  $(z_n)_{n\geq 0}$ is a sequence of positive reals with $z_0>0$ satisfying $\limsup_{n\to\infty} e^{\phi(n)}z_n < \infty$,  then
\begin{equation}\label{eq:badgood}
	 \liminf_{k\rightarrow\infty}\frac{1}{k}\# \left\{1\leq r \leq k: \text{$z$ is $\psi$-good on scale $a(r)$}\right\}\geq \frac{1}{2}.
\end{equation}	
\end{lemma}

As mentioned above, the function $\psi$, which we think of as ``$\phi$ with some room'', is used to define the threshold for a drop on scale $k$ to be ``large'', with $z$ being $\psi$-good on scale $k$ precisely when a good proportion of the total decay on this scale comes from drops that are larger than this threshold. Meanwhile, the sequence $a$ is used to take a sparse sequence of spatial scales so that we can safely ignore dependencies between scales while keeping various series divergent so that we can still hope to conclude via Borel-Cantelli.

The proof of Lemma \ref{lem:decaydensity} will rely on the following elementary analytic facts.

\begin{lemma} \label{lem:subseq}
	Suppose that $f:\N\rightarrow[0,\infty)$ is a decreasing function satisfying
	$\sum_{n=1}^\infty f(n)=\infty$. 
	\begin{enumerate}
		\item  If $A\subseteq \N$ has positive density in the sense that $\liminf_{N\to\infty} \frac{1}{N} \sum_{n=1}^N \mathbbm{1}(n \in A)>0$ then 
		\[
		\sum_{n\in A} f(n)=\infty \qquad \text{ and } \qquad \liminf_{N\to\infty}\frac{\sum_{n=1}^N \mathbbm{1}(n \in A) f(n)}{\sum_{n=1}^N f(n)}>0.
		\]
		\item There exists a convex, strictly increasing function $a:\N\rightarrow\N$ with $\lim_{n\rightarrow\infty}a(n)-a(n-1)=\infty$ such that $\sum_{n=1}^\infty f(a(n))=\infty$.
	\end{enumerate} 
\end{lemma}

\begin{proof}[Proof of Lemma \ref{lem:subseq}]
	Fix $f$ as in the statement of the lemma.
	We begin with the first statement. Let $A\subseteq \N$ be such that $\liminf_{N\to\infty} \frac{1}{N} \sum_{n=1}^N \mathbbm{1}(n \in A)>0$ and let $k$ be such that $\liminf_{N\to\infty} \frac{1}{N} \sum_{n=1}^N \mathbbm{1}(n \in A) > 2/k$, so that there exists $\ell_0$ such that $\sum_{n=k^{\ell-1}}^{k^\ell+1} \mathbbm{1}(n \in A) \geq 2 k^{\ell-1}-k^{\ell-1}= k^{\ell-1}$ for every $\ell \geq \ell_0$. Letting $N\geq k^{\ell_0+1}$ and setting $\ell_1 = \lfloor \log_k N\rfloor$, we have that
	\[\sum_{n=k^{\ell_0}+1}^{N} \mathbbm{1}(n\in A)f(n) \geq \sum_{\ell=\ell_0}^{\ell_1-1} f(k^{\ell+1}) \sum_{n=k^{\ell}+1}^{k^{\ell+1}}\mathbbm{1}(n\in A) \geq   \sum_{\ell=\ell_0}^{\ell_1-1} k^\ell f(k^{\ell+1})\]
and that
	\[
\sum_{n=k^{\ell_0}+1}^{N} f(n) \leq \sum_{\ell=\ell_0}^{\ell_1} k^{\ell+1} f(k^{\ell}) \leq k^{\ell_0+1} f(k^{\ell_0}) + k^2 \sum_{\ell=\ell_0}^{\ell_1-1} k^\ell f(k^{\ell+1}).
	\]
	Since $f$ was assumed to be divergent, it follows that 
	\[
\liminf_{N\to\infty}\frac{\sum_{n=1}^N \mathbbm{1}(n \in A) f(n)}{\sum_{n=1}^N f(n)}\geq \frac{1}{k^2}>0,
	\]
	and hence that $\sum_{n\in A} f(n)=\infty$ as claimed.
	
	We now prove the second statement. The first statement implies that for any $a,b\geq 1$ there exists $m=m(a,b)\geq 1$ such that
	$\sum_{n=1}^m f(a+bn)\geq1$.
This fact allows us to recursively construct a pair of integer sequences $(b_i)_{i\geq 0}$ and $(d_i)_{i\geq 0}$ by setting $b_0=0$ and recursively defining 
	\[d_i = \min\Bigl\{m:\sum_{ n=1}^mf(b_{i}+2^in)\geq1\Bigr\} \qquad \text{ and } \qquad b_{i+1}=b_i+2^i d_i \qquad \text{ for each $i\geq 0$},\]
	and we observe that both $d_i$ and $b_i$ must be finite for $i\geq 0$.
	We must then have
	\[
	\sum_{i=1}^\infty\sum_{n=1}^{d_i}f(b_{i}+2^in) = \sum_{n = 1}^\infty f(a(n)) = \infty,
	\]
	where $a(n)$ is the convex, strictly increasing sequence defined by
	\[
	(a(1),a(2),\ldots) = (b_1+2,b_1+2\cdot 2,\ldots,b_1+2 \cdot d_1,b_2+2^2,b_2+2^2\cdot 2,\ldots,b_2+2^2 \cdot d_2,b_3+2^3, b_3 + 2^3 \cdot 2,\ldots)\]
	This sequence has increasing increments tending to infinity by construction, completing the proof.
\end{proof}

We now apply Lemma \ref{lem:subseq} to prove Lemma \ref{lem:decaydensity}.

\begin{proof}[Proof of Lemma \ref{lem:decaydensity}]
We may assume without loss of generality that $\phi(x) \leq \sqrt{x}$ for every $x\geq 0$, replacing $\phi$ with $ \tilde \phi=\min\{\phi,\sqrt{x}\}$ otherwise. Indeed, since $\phi$ is increasing, we have that
$\tilde \phi^{-1}(y) \leq \max\{\phi^{-1}(y),y^2\}$, and we have by the Cauchy condensation test that
\[
\sum_{k=1}^\infty \frac{1}{1 \vee \log \max\{\phi^{-1}(k),k^2\}} = \infty \quad \text{ if and only if }\quad \sum_{k=1}^\infty \frac{2^k}{1\vee \max\{\log \phi^{-1}(2^k),k\log 4\}} = \infty.
\]
If there are infinitely many $k$ such that $4^k \geq \phi^{-1}(2^k)$ then the right hand series trivially diverges, while if not then it diverges as a consequence of the Cauchy condensation test applied to $\sum_{k=1}^\infty \frac{1}{1\vee\log \phi^{-1}(k)}$.

	We begin by applying Lemma \ref{lem:subseq} to the function $f(k)=1/1\vee\log\phi^{-1}(k)$ to give a strictly increasing function $a(k):\N\rightarrow\N$ such that $\lim_{k\rightarrow\infty} a(k)-a(k-1)=\infty$ and 
	\begin{equation}\label{eq:lemmaapp}
	\sum_{k=1}^\infty \frac{1}{1 \vee \log\phi^{-1}(a(k))}=\infty.
	\end{equation} 
Extend $a$ arbitrarily to an increasing bijection $a:[0,\infty)\to [0,\infty)$ and 
define $\psi:[0,\infty)\to[0,\infty)$ to be the inverse of the increasing bijection
\[
\psi^{-1}(x) =8\phi^{-1}\Bigl(a\left(8a^{-1}(x)\right)\Bigr),
\]
so that $\psi$ is strictly increasing, bounded above by $\phi$, and satisfies
	\[
\sum_{k=1}^\infty \frac{1}{1\vee\log(\psi^{-1}(a(k)))} = \sum_{k=1}^\infty \frac{1}{1\vee(\log 8+\log(\phi^{-1}(a(8k))))}=\infty
	\]
	by Lemma \ref{lem:subseq}. Since $\phi$ and $a$ are increasing and $\phi(x) \leq \sqrt{x}$ for every $x\geq 0$ we also have that $\psi^{-1}(x)\geq 8 x^2$ and $\psi(x)\leq \sqrt{x}$ for every $x \geq 0$.

	Let $(z_n)_{n\geq 0}$ be a sequence of positive numbers and let $C$ be such that $z_n\leq Ce^{-\phi(n)}$ for every $n\geq 1$. We will use the notation of 
	Definition \ref{def:drops}. 
	Observe that if $k\geq k_0$ is such that $z$ is \emph{not} $\psi$-good on scale $k$ then we must have that
	\[
\#\left\{i:\frac{d_{i+1,k}}{d_{i,k}}>\exp\left[{-\frac{1}{2\psi^{-1}(k)}}\right]\right\}>\frac{\log e^{-1/2}}{\log \exp(-k/(2\psi^{-1}(k))}=\frac{1}{k}\psi^{-1}(k).
	\]
	Discounting the endpoints of the interval, it follows that
	\[
	\#\bigl(\{z_m^*: m \geq 0\}\cap (e^{-k-1},e^{-k})\bigr)>\frac{1}{k}\psi^{-1}(k)-2 \geq \frac{1}{2k}\psi^{-1}(k)
	\]
	whenever $z$ is not $\psi$-good on scale $k$, where we used that $\psi^{-1}(k)\geq 8k^2$ in the final inequality.
	Let $B$ be the set of positive integers $k\geq k_0$ such that $z$ is not $\psi$-good on scale $a(k)$ and let
	\[
	A=\left\{k \geq k_0 : \frac{1}{k}|B \cap \{k_0,\ldots,k\}| \geq \frac{1}{2} \right\}.
	\]
	We wish to prove that $A$ is finite, and observe that $A$ is finite if and only if $A \cap B$ is finite.
	For each $k \in A \cap B$ with $k \geq 4 k_0$, we have that $|B \cap \{\lfloor k/4 \rfloor, \ldots k\}| \geq k/4$ and hence, since $\psi^{-1}$ is increasing, that
	\begin{align*}|\{n : z_n \geq e^{-a(k)-1}\}| 
	&\geq \sum_{i\in B\cap\{k_0,\ldots,k\}} \abs{\{z_m^*\}\cap (e^{-a(i)-1},e^{-a(i)})}
	\\&\geq \frac{1}{2}\sum_{i\in B\cap\{k_0,\ldots,k\}} \frac{1}{k}\psi^{-1}(a(i)) \geq \frac{1}{8} \psi^{-1}(a(\lfloor k/4 \rfloor)).
	\end{align*}
As such, for each $k \in A \cap B$ with $k \geq 4 k_0 \geq 4$, there must exist $n\geq \frac{1}{8} \psi^{-1}(a(\lfloor k/4 \rfloor))$ such that $z_n \geq e^{-a(k)-1}$.
On the other hand, for such $n$, we also have that
\[
\phi(n) \geq \phi\left(\frac{1}{8} \psi^{-1}(a(\lfloor k/4 \rfloor)) \right) = a\left(8\lfloor k/4 \rfloor\right) \geq a(2k),
\]
and since $a(2k)-a(k) \to\infty$ and $\limsup_{n\to\infty} e^{\phi(n)} z_n < \infty$, we deduce that $A \cap B$ is finite as claimed.
\end{proof}

\subsection{Proof of Theorem \ref{thm:kcuts}}
\label{subsec:mainproof}

We now apply the deterministic tools from Section \ref{subsec:deterministic} to prove Theorem \ref{thm:kcuts}. Let us fix for the remainder of this subsection a transient Markov chain with killing $M=(\Omega,P,\dagger)$ with distinguished origin vertex $o$ such that $M$ is irreducible, locally finite, and satisfies $P(x,x)=0$ for every $x\neq \dagger$. Fix $X_0 \in \Omega \setminus \{\dagger\}$, let $(\mathcal{X}_t)_{t\geq 0}$ be the drawbridge process started at $(X_0,X_0)$, let $(X_n)_{n\geq 0}$ be the associated discrete-time trajectory of the Markov process, and let $(\mathcal{Z}_t)_{t\geq 0}$ and $(Z_n)_{n\geq 0}$ be the associated continuous- and discrete-time hitting probability processes as defined in \cref{subsec:drawbridge}. We write $\mathbf{P}$ and $\mathbf{E}$ for probabilities and expectations taken with respect to the joint law of these processes.

Fix an increasing bijection $\phi:[0,\infty)\to[0,\infty)$ as in the statement of the theorem, and let $a$ and $\psi$ be as in Lemma \ref{lem:decaydensity}. For each $k\geq 1$ and $0\leq i \leq N_k-1$ we say that the drop $(d_{i,k},d_{i+1,k})$ is a \textbf{large drop} if
\[
\frac{d_{i+1,k}}{d_{i,k}} \leq \Psi(k):=\exp\left[-\frac{k}{2\psi^{-1}(k)}\right]
\]
and there exists $n$ such that $Z_n^*=d_{i+1,k}$. (The latter condition can fail if $d_{i+1,k}=e^{-k-1}$.) If $(d_{i,k},d_{i+1,k})$ is a large drop and $Z$ first hits $d_{i+1,k}$ at some time $n$, we say that $(d_{i,k},d_{i+1,k})$
is a \textbf{large permadrop} if we have additionally that
\[
\mathcal{Z}_t < d_{i,k} \text{ for every $t \geq \mathcal{T}(n)$.}
\]
We say that an arbitrary pair of values $(a,b)$ in $[e^{-k-1},e^{-k}]$ with $a>b$ is a large drop or large permadrop if $(a,b)=(d_{i,k},d_{i+1,k})$ for some large drop or large permadrop $(d_{i,k},d_{i+1,k})$ as appropriate.

Given $k \geq k_0$ and $i\geq 1$, we write $\tau_{i}=\tau_{i,k}$ for the $i$th time the discrete process $Z$ reaches a new running minimum smaller than $e^{-k}$, and write $\mathcal{T}_i=\mathcal{T}(\tau_{i,k})$ for the corresponding time for the continuous process $\mathcal{Z}$, noting that $\mathcal{T}_{i,k}$ is a stopping time for $\mathcal{X}$ for each $i,k\geq 1$. Note that when $1\leq i \leq N_k$, $\tau_{i,k}$ can be defined equivalently as the first time that $Z_n \leq d_{i,k}$. 
Recall that if $\rho$ is a stopping time for $\mathcal{X}$ then $\mathcal{F}_\rho$ denotes the $\sigma$-algebra generated by $(\mathcal{X}_t)_{t=0}^{\rho}$.
 Given such a stopping time, we lighten notation by writing $\mathbf{E}_\rho[{}\cdot{}]=\mathbf{E}[\ \cdot\mid\mathcal{F}_\rho]$ and $\mathbf{P}_\rho[{}\cdot{}]=\mathbf{P}[\ \cdot\mid\mathcal{F}_\rho]$.

The following two estimates on the distribution of the random variable
$R_k := \#\{0\leq i \leq N_{k}-1 : (d_{i,k},d_{i+1,k}) \text{ is a large permadrop}\}$
lie at the heart of the paper. 

\begin{proposition} \label{prop:k_lower} The estimate
	\begin{equation} \label{eq:UBN}
		\mathbf{E}_{\mathcal{T}_{1,k}}[R_k]\geq \frac{1}{4} \mathbf{P}_{\mathcal{T}_{1,k}}\Bigl(\text{\emph{$Z$ is $\psi$-good on scale $k$ and there exists $n$ such that $e^{-k-1} < Z_n^* \leq e^{-k-3/4}$}}\Bigr)
	\end{equation}
	holds almost surely for every $k\geq 1$.
\end{proposition}

\begin{proposition} \label{prop:k}
There exists a universal constant $C$ such that the estimate
	\begin{equation} \label{eq:LBN}
		\mathbf{E}_{\mathcal{T}_{1,k}}[R_k]\leq \left(2+C\log\frac{2\psi^{-1}(k)}{k}\right)\mathbf{P}_{\mathcal{T}_{1,k}} (R_k\geq 1)
	\end{equation}
	holds almost surely for every $k\geq 1$.
\end{proposition}

\begin{proof}[Proof of Proposition \ref{prop:k_lower}]
To lighten notation, we drop $k$s from subscripts wherever possible. We can use the optional stopping theorem to compute
\begin{align}
		\mathbf{E}_{\mathcal{T}_1}[R]&\geq
		\mathbf{E}_{\mathcal{T}_1} 
		\left[ \sum_{i \geq 0} \mathbbm{1}(Z_{\tau_{i+1}} > e^{-k-1})\mathbf{P}_{\mathcal{T}_{i+1}} ((d_i,d_{i+1}) \text{ is a large permadrop}) \right]
		\nonumber\\  &= \mathbf{E}_{\mathcal{T}_1} \left[\sum_{i=0}^{N_{k}-2} \left( 1-\frac{d_{i+1}}{d_i}\right)\mathbbm{1}\left(\frac{d_{i+1}}{d_i}\leq \Psi\right)\right]
		\label{eq:OptionalStopping1}
\end{align}
and applying the inequality $1-x\geq\log x$ yields that
\begin{align}
		\mathbf{E}_{\mathcal{T}_1}[R]&\geq \mathbf{E}_{\mathcal{T}_1}\left[\log \prod\left\{\frac{d_{i}}{d_{i+1}}:0\leq i\leq N_{k}-2,\frac{d_{i+1}}{d_i}\leq \Psi\right\}\right].
		\label{eq:Elogprod}
	\end{align}
	When $Z$ is $\psi$-good on scale $k$ and there exists $n$ such that $e^{-k-1} < Z_n^* \leq e^{-k-3/4}$ we have that
	\[
\prod\left\{\frac{d_{i}}{d_{i+1}}:0\leq i\leq N_{k}-1,\frac{d_{i+1}}{d_i}\leq \Psi\right\} \geq e^{1/2} \qquad \text{ and } \qquad \frac{d_{N_k-1}}{d_{N_k}} \leq e^{1/4},
	\]
so that the claim follows from \eqref{eq:Elogprod}.
\end{proof}

\begin{proof}[Proof of Proposition \ref{prop:k}]
	We fix a scale $I_k=[e^{-k-1},e^{-k}]$ and calculate the conditional expectation of the number of large permadrops on that scale given there is at least one.  Since $k\geq 1$ is fixed throughout, we will drop it from notation when possible. We condition on the location on the first permadrop being $(a,b)\subset I_k$ as follows.
	Let $R^\prime$ be the number of large permadrops in the scale excluding possibly the last drop, so that
	\[
		R^\prime=\sum_{0\leq i\leq N_k-2} \mathbbm{1}((d_i,d_{i+1}) 
		\text{ is a large permadrop}),
	\] 
	and let $R''=(R'-1)\vee 0$ be the amount that $R'$ exceeds $1$.
	  Given an arbitrary pair $ e^{-k-1}\leq b<a\leq e^{-k}$, we say that $(a,b)$ is the \textbf{first large permadrop} if $(a,b)=(d_i,d_{i+1})$ for some $0\leq i \leq N_k-1$ such that the pair $(d_i,d_{i+1})$ is a large permadrop and $(d_j,d_{j+1})$ is not a large permadrop for any $j<i$. We write $\tau_b$ for the first time $Z_n$ hits $b$ (letting $\tau_b=\infty$ if this never occurs), write $\mathcal{T}_b=\mathcal{T}(\tau_b)$, and seek to upper bound the conditional expectation
	  \[
\mathbf{E}_{\mathcal{T}_b}\left[\mathbbm{1}((a,b) \text{ is the first large permadrop}) R'' \right].
	  \]
	  If $b= e^{-k-1}$ then this conditional expectation is zero, so assume $b> e^{-k-1}$.

	    Note that if a large drop $(d_i,d_{i+1})$ is \emph{not} a permadrop then there must exist a \textit{recovery time} at which $\mathcal{Z}$ hits $d_i$ for the first time after $\mathcal{T}_{i+1}$. Let $L=L(a,b)=\{(d_i,d_{i+1}):1\leq i \leq N_k-2$, $d_i\leq b$, and $d_{i+1}/d_i \leq \Psi(k)\}$ be the set of large drops on the scale $k$ after $(a,b)$ and possibly excluding the last drop, let $K=K(a,b)$ be the event that $(a,b)$ is a drop and that every previous drop in the scale recovers before time $\mathcal{T}_b$, and observe that 
	   \begin{multline*}
	   \mathbbm{1}((a,b) \text{ is the first large permadrop}) R'' \\= \mathbbm{1}(K)\mathbbm{1}((a,b) \text{ is a large permadrop}) \sum_{(d_i,d_{i+1})\in L} \mathbbm{1}((d_i,d_{i+1}) \text{ is a large permadrop}).
	   \end{multline*}
	   Since $\mathcal{T}_i$ is a stopping time for each $i\geq 1$, we can use the optional stopping theorem to compute that if $b/a\leq \Psi(k)$ then
	\begin{align}
		&\mathbf{E}_{\mathcal{T}_b} \left[\mathbbm{1}((a,b)\text{ is the first large permadrop})R^{\prime\prime}\right]
		\nonumber\\&\hspace{3cm}= \mathbf{E}_{\mathcal{T}_b}\biggl[\mathbbm{1}(K)\sum_{(d_i,d_{i+1})\in L}\mathbbm{1}((a,b) \text{ and } (d_i,d_{i+1}) \text{ are permadrops})\biggr] 
		\nonumber\\
		&\hspace{3cm}= \mathbf{E}_{\mathcal{T}_b}\bigg[\mathbbm{1}(K)\sum_{(d_i,d_{i+1})\in L}\mathbbm{1}(\mathcal{Z}_t<a \text{ for }t\in(\mathcal{T}_b,\mathcal{T}_{i+1}) \text{ and }\mathcal{Z}_t<d_i \text{ for } t>\mathcal{T}_{i+1})\bigg]
		\nonumber\\
		&\hspace{3cm}= \mathbf{E}_{\mathcal{T}_b}\bigg[\mathbbm{1}(K)\sum_{(d_i,d_{i+1})\in L}\bigl(1-\frac{d_{i+1}}{d_i}\bigr)\mathbbm{1}\bigl(\mathcal{Z}_t<a \text{ for }t\in(\mathcal{T}_b,\mathcal{T}_{i+1})\bigr) \bigg] 
		\nonumber\\ &\hspace{3cm}=
		\mathbf{E}_{\mathcal{T}_b}\bigg[\mathbbm{1}(K)\sum_{(d_i,d_{i+1})\in L}\big(1-\frac{d_{i+1}}{d_i}\big)\mathbbm{1}((\mathcal{Z}_t)_{t>\mathcal{T}_b} \text{ hits }d_{i+1}\text{  before } a ) \bigg],
		\label{eq:first_opt}
	\end{align}
	where the third equality follows by taking conditional expectations with respect to $\mathcal{F}_{\mathcal{T}_{i+1}}$ for the term in the summation corresponding to $(d_i,d_{i+1})$ and then applying optional stopping.

Write $\Psi=\Psi(k)$ and let $\mathscr{L}=\mathscr{L}(b)$ be the set of all finite sets $S$ of ordered pairs  of numbers in $[e^{-k-1},e^{-k}]$ satisfying the following conditions:
\begin{enumerate}
	\item If $(x,y) \in S$ then $x\leq b$ and $y/x \leq \Psi$.
	\item If $(x,y)$ and $(z,w)$ are distinct elements of $S$ then the open intervals $(y,x)$ and $(w,z)$ are disjoint.
\end{enumerate}
 If we consider the (random) function $F:\mathscr{L}\to[0,\infty)$ defined by
	\[
	F(S)=\sum_{(x,y)\in S}\big(1-\frac{y}{x}\big)\mathbbm{1}((\mathcal{Z}_t)_{t>\mathcal{T}_b} \text{ hits }y\text{ before } a ),
	\]
	then we can rewrite \eqref{eq:first_opt} as
	\begin{equation}\label{eq:rewritten_in_terms_of_F}
		\mathbf{E}_{\mathcal{T}_b}\left[\mathbbm{1}((a,b)\text{ is the first permadrop})R_k^{\prime\prime}\right]\leq \mathbf{E}_{\mathcal{T}_b}\left[\mathbbm{1}(K)F(L)\right],
	\end{equation}
and we  claim that the inequality
	\begin{equation} \label{eq:deterministic_inequality}
	F(S)\leq (1-\Psi^2) \sum_{i=1}^{\lceil-1/\log \Psi\rceil}\mathbbm{1}((\mathcal{Z}_t)_{t>\mathcal{T}_b} \text{ hits }\Psi^i b\text{ before } a ),
	\end{equation}
	is satisfied deterministically for every $S\in \mathscr{L}$. 

	Before proving the claimed inequality \eqref{eq:deterministic_inequality}, let us first see how it implies \eqref{eq:LBN}. 
	Substituting \eqref{eq:deterministic_inequality} into \eqref{eq:rewritten_in_terms_of_F} yields that
	\begin{align*}
	\mathbf{E}_{\mathcal{T}_b} \left[\mathbbm{1}((a,b)\text{ is the first permadrop})R^{\prime\prime}\right]
	&\leq (1-\Psi^2)\mathbf{E}_{\mathcal{T}_b}{\mathbbm{1}(K)\sum_{i=1}^{\lceil-1/\log \Psi\rceil}\mathbbm{1}((\mathcal{Z}_t)_{t>\mathcal{T}_b} \text{ hits }\Psi^i b\text{ before } a )}
	\\
	&= (1-\Psi^2)\mathbf{E}_{\mathcal{T}_b}\left[\mathbbm{1}(K)\sum_{i=1}^{\lceil-1/\log \Psi\rceil} \frac{a-b}{a-\Psi^i b}\right] 
	\\
	&= (1-\Psi^2) \frac{a-b}{a} \mathbbm{1}(K) \sum_{i=1}^{\lceil-1/\log \Psi\rceil} \frac{1}{1-(b/a)\Psi^i },
	\end{align*}
	where we have applied optional stopping in the second inequality. 
	Since $1/(1-x\Psi^i)$ is an increasing function of $x$ it follows that if $b\leq \Psi a$ then
	\begin{align*}
	\mathbf{E}_{\mathcal{T}_b} \left[\mathbbm{1}((a,b)\text{ is the first permadrop})R^{\prime\prime}\right]
	&\leq \frac{a-b}{b}\mathbbm{1}(K) (1-\Psi^2) \sum_{i=1}^{-\lceil 1/\log \Psi \rceil} \frac{1}{1-\Psi^{i+1}}.
	\end{align*}
	Now, we have by calculus that $1-\Psi^{i+1} \geq (i+1) (1-\Psi)$ for every $1\leq i \leq -\lceil 1/\log \Psi\rceil$ and hence that there exists a universal constant $C$ such that
	\[\mathbf{E}_{\mathcal{T}_b} \left[\mathbbm{1}((a,b)\text{ is the first permadrop})R^{\prime\prime}\right]
	\leq
	 \frac{a-b}{b}\mathbbm{1}(K) \frac{1-\Psi^2}{1-\Psi} \sum_{i=1}^{-\lceil 1/\log \Psi \rceil} \frac{1}{i+1} \leq C\frac{a-b}{b} \mathbbm{1}(K) \log \frac{2\psi^{-1}(k)}{k},
	\]
	where we used the definition of $\Psi=\Psi(k)$ in the final inequality. Since we also have by optional stopping that
	\begin{equation}
	\mathbf{P}_{\mathcal{T}_b} \left((a,b)\text{ is the first permadrop})\right) = \frac{a-b}{b} \mathbbm{1}(K)
	\label{eq:optional_stopping_equality},
	\end{equation}
	we can take expectations over $\mathcal{F}_{\mathcal{T}_b}$ conditional on $\mathcal{F}_{\mathcal{T}_1}$ to deduce that
	\begin{multline}
\mathbf{E}_{{\mathcal{T}_1}} \left[\mathbbm{1}((a,b)\text{ is the first large permadrop})R^{\prime\prime}_k \right] \\ \leq C \log \left[ \frac{2\psi^{-1}(k)}{k} \right] \cdot \mathbf{P}_{{\mathcal{T}_1}} \left((a,b)\text{ is the first large permadrop}\right).
	\end{multline}
	We stress that this holds for \emph{every} pair of real numbers $a>b$ in $[e^{-k-1},e^{-k}]$, but that both sides will be equal to zero for all but countably many such pairs. Summing over the countable set of pairs giving a non-zero contribution yields that
	\begin{equation*} 
	\mathbf{E}_{\mathcal{T}_1}[R]\leq 2\mathbf{P}_{{\mathcal{T}_1}}(R\geq 1) + \mathbf{E}_{\mathcal{T}_1}[R'']\leq \left(2+C\log\frac{2\psi^{-1}(k)}{k}\right)\mathbf{P}_{\mathcal{T}_1} (R\geq 1)
\end{equation*}
as claimed.

It remains to prove \eqref{eq:deterministic_inequality}.
Given a set $S\in\mathscr{L}$, we say $S$ is \textbf{slack} if there exists an element $(x,y)\in S$ such that $y/x < \Psi^2$ and \textbf{taut} otherwise. Observe that if $S\in \mathscr{L}$ is slack and $(x,y)\in S$ satisfies $y/x < \Psi^2$ then the set $S' = S \cup \{(x,\Psi x),(\Psi x, y)\} \setminus \{(x,y)\}$ also belongs to $\mathscr{L}$ and satisfies $F(S)\leq F(S')$. Indeed, the latter inequality follows from the pointwise inequality
	\begin{multline*}
		\mathbbm{1}((\mathcal{Z}_t)_{t>\mathcal{T}_b} \text{ hits }y_{i}\text{ before } a ) \Bigl(1-\frac{y_i}{x_i}\Bigr)\\
		\leq\mathbbm{1}((\mathcal{Z}_t)_{t>\mathcal{T}_b} \text{ hits }\Psi x_i\text{ before } a ) \Bigl(1-\frac{{\Psi x_i}}{x_i}\Bigr)+\mathbbm{1}((\mathcal{Z}_t)_{t>\mathcal{T}_b} \text{ hits }y_{i}\text{ before } a ) \Bigl(1-\frac{y_i}{\Psi x_i}\Bigr).
	\end{multline*}
	To verify this inequality, note that if the indicator on the left is one, then so are both indicators on the right, and when all three indicators are equal to one, the inequality is equivalent to the elementary inequality 
	\[
	1-\frac{\Psi x_i}{x_i} + 1 - \frac{y_i}{\Psi x_i} -\left(1-\frac{y_i}{x_i}\right)= \frac{\Psi(1-\Psi) x_i - (1- \Psi) y_i}{\Psi x_i} \geq  \frac{\Psi(1-\Psi) x_i - (1- \Psi) \Psi^2 x_i}{\Psi x_i} = (1-\Psi)^2\geq 0
	\]
	which holds since $y_i< \Psi^2 x_i$.
	Given a slack set $S \in \mathscr{L}$, we can therefore iterate this operation until we obtain a taut set $S^\bullet$ with $F(S)\leq F(S^\bullet)$; this iterative process must terminate after finitely many steps since $|S'|=|S|+1$ and every set in $\mathscr{L}$ contains at most $\lceil -1/\log \Psi \rceil$ pairs of points. Enumerate the pairs of points of $S^\bullet$ in decreasing order as $(x_1,y_1)$, $(x_2,y_2)$, \ldots, $(x_\ell,y_\ell)$. Since every pair $(x,y) \in S^\bullet$ satisfies $y/x\leq \Psi^2$ and every two distinct pairs of points in $S^\bullet$ span disjoint open intervals of $[e^{-k-1},b]$ we must have that $y_i \leq \Psi^i b$ for every $1 \leq i \leq \ell$ and hence that $\ell \leq \lceil -1/\log \Psi\rceil$ as previously mentioned.
	It follows that
	\begin{multline*}
F(S)\leq F(S^{\sbt})=\sum_{i=1}^\ell \mathbbm{1}((\mathcal{Z}_t)_{t>\mathcal{T}_b} \text{ hits }y_i\text{ before } a ) \big(1-\frac{y_i}{x_i}\big)
\\\leq \sum_{i=1}^{\ceil{-1/\log\Psi}}\mathbbm{1}((\mathcal{Z}_t)_{t>\mathcal{T}_b} \text{ hits }\Psi^i b\text{ before } a ) (1-\Psi^2) ,
\end{multline*}
as claimed, where we used that $S^\bullet$ is taut in the second inequality.
\end{proof}

With these bounds in hand, we can now prove Theorem \ref{thm:kcuts}. 
The proof will apply the Borel-Cantelli counterpart \cite{MR587211} which is an extension of the second Borel-Cantelli lemma to dependent events.
\begin{lemma}[Borel-Cantelli Counterpart] \label{lem:BCC}
	If $(E_n)_{n\geq0}$ is an increasing sequence of events satisfying the divergence condition
	$\sum_{n\geq 1} \pr(E_n\mid E_{n-1}^c)=\infty$, then $\pr(\bigcup_{n\geq 1} E_n)=1$.
\end{lemma}
Setting $E_n=\cup_{1\leq i\leq n} A_i$ for $n\geq 1$ where $(A_i)_{i\geq 1}$ is an arbitrary sequence of events,
the Borel-Cantelli counterpart implies in particular that
\begin{equation}
\label{eq:counterpart}
\text{ if } \qquad \sum_{n\geq 1} \pr\big(A_n\mid (\cup_{i=1}^{n-1} A_i)^c\big)=\infty \qquad \text{ then } \qquad \pr\Biggl(\,\bigcup_{n\geq 1} A_n \Biggr)=1.
\end{equation}

\begin{proof} [Proof of Theorem  \ref{thm:kcuts}]
	We assume that $\tau_\dagger=\infty$ with positive probability, the claim holding vacuously otherwise.
	 We continue to use $\phi$, $\psi$, $a$, and $k_0$ as defined above and let $\mathscr{G}$ be the event that $Z_n \leq e^{-\phi(n)}$ for all sufficiently large $n$. Since each permadrop gives rise to a distinct cut time, it suffices to prove that $\sum_{k} R_k = \infty$ almost surely on the event that $Z$ is not killed.
	For each $k\geq 1$, let $A_k=\{R_{a(k)}\geq 1\}$ be the event that there is at least one large permadrop on the $a(k)$th scale, let $B_k$ be the event that there exists $n$ such that $e^{-a(k)-1} < Z_n^* \leq e^{-a(k)-3/4}$, and let $C_k$ be the event that there exists $n$ such that $\sqrt{e^{-a(k-1)-1}e^{-a(k)}} \leq Z_n^* < a^{-a(k-1)-1}$
	. 
	If $(B_k \cap C_k)^c$ holds for infinitely many $k$ then $Z$ is either killed or satisfies $Z_{n+1}^* \leq e^{-1/4} Z_n^*$ for infinitely many $n$, in which case the claim follows from the proof of Proposition~\ref{prop:expDeccase}. As such, it suffices to prove that 
	\[
\mathbf{P}\biggl( \bigcup_{k=k_1}^\infty A_k \mid \mathscr{G} \text{ holds, } \tau_\dagger = \infty, \text{ and } B_k \cap C_k \text{ holds for all sufficiently large $k$}\biggr) =1 \text{ for every $k_1 \geq k_0$}
	\]
    whenever the event being conditioned on has positive probability.
	We will assume for contradiction that there exists $k_1 \geq k_0$ such that this does not hold, and fix such a $k_1$ for the remainder of the proof. 
	
	For each $k\geq k_0$, let $G_k$ be the event that $Z$ is $\psi$-good on scale $a(k)$. It follows from Propositions \ref{prop:k_lower} and \ref{prop:k} that there exists a universal constant $c>0$ such that
	\begin{equation} \label{eq:propApplication}
\mathbf{P}_{\mathcal{T}_{1,a(k)}} (A_k) \geq  \frac{c \cdot \mathbf{P}_{\mathcal{T}_{1,a(k)}}(G_k \cap B_k)}{1+ \log \psi^{-1}(a(k)) }
	\end{equation}
	for every $k\geq 1$, where we used that $\psi(x)\leq \sqrt{x}$ to bound $\log (\psi^{-1}(a(k))/a(k)) \geq \frac{1}{2}\log \psi^{-1}(a(k))$. For each $k\geq k_0$, let $F_k$ be the event that 
	for each $\ell < k$ and $0 \leq i \leq N_{a(\ell)-1}$ such that the drop $(d_{i,a(\ell)},d_{i+1,a(\ell)})$ is \emph{not} a permadrop, the process $\mathcal{Z}$ hits $d_{i,a(\ell)}$ at a time between $\mathcal{T}_{i+1,a(\ell)}$ and $\mathcal{T}_{1,a(k)}$.
The event $F_k$ is constructed precisely to force decorrelation between $A_k$ and $(\bigcup_{i=k_1}^{k-1} A_i)^c$. Indeed, the intersection $F_k \cap C_k \setminus \cup_{i=k_1}^{k-1} A_i$ is measurable with respect to $\mathcal{F}_{\mathcal{T}_{1,a(k)}}$ 
and we can apply \eqref{eq:propApplication} to deduce that
\begin{align}
\mathbf{P}(A_k \cap F_k \cap C_k \setminus \cup_{i=k_1}^{k-1} A_i)
&= \mathbf{E}\left[\mathbbm{1}(F_k \cap C_k \setminus \cup_{i=k_1}^{k-1} A_i) \mathbf{P}_{\mathcal{T}_{1,a(k)}} (A_k) \right] \nonumber
\\&\geq \frac{c}{1+ \log \psi^{-1}(a(k)) } \mathbf{E}\left[\mathbbm{1}(F_k \cap C_k \setminus \cup_{i=k_1}^{k-1} A_i) \mathbf{P}_{\mathcal{T}_{1,a(k)}} (G_k \cap B_k) \right] \nonumber\\&= \frac{c}{1+ \log \psi^{-1}(a(k)) } \mathbf{P}(F_k \cap C_k \cap G_k \cap B_k \setminus \cup_{i=k_1}^{k-1} A_i)
\end{align}
for every $k\geq k_0$. On the other hand, it follows by optional stopping that
\[
\mathbf{P}(F_k^c \cap C_k \cap G_k \cap B_k \setminus \cup_{i=k_1}^k A_i) \leq \mathbf{P}(F_k^c \cap C_k) \leq \sqrt{e^{-a(k)+a(k-1)+1}}
\]
and hence that
\begin{align}
\mathbf{P}(A_k \setminus \cup_{i=k_1}^{k-1} A_i)&\geq \frac{c}{1+ \log \psi^{-1}(a(k)) } \left(\mathbf{P}(C_k \cap G_k \cap B_k \setminus \cup_{i=k_1}^{k-1} A_i)-\sqrt{e^{-a(k)+a(k-1)+1}}\right)\vee 0.
\end{align}
We deduce by linearity of expectation that
\begin{equation}
\sum_{k=k_1}^\ell \mathbf{P}(A_k \setminus \cup_{i=k_1}^{k-1} A_i) \geq c \cdot \mathbf{E}\left[ \sum_{k=k_1}^{\ell} \frac{(\mathbbm{1}(C_k \cap G_k \cap B_k \setminus \cup_{i=k_1}^{k-1} A_i)-\sqrt{e^{-a(k)+a(k-1)+1}})\vee 0}{1+ \log \psi^{-1}(a(k)) }  \right].
\label{eq:nearly_there}
\end{equation}
On the event that $\mathscr{G}\setminus \cup_{i=k_1}^\infty A_i$ holds and $B_k \cap C_k$ holds for all sufficiently large $k$ (which has positive probability by assumption), we have by choice of $\psi$ in Lemma~\ref{lem:decaydensity} that $\liminf_{k\to\infty}\frac{1}{k}\sum_{\ell =k_1}^k \mathbbm{1}(C_k \cap G_k \cap B_k \setminus \cup_{i=k_1}^\infty A_i) >0$ and hence by Lemma \ref{lem:subseq} that there exists an almost surely positive $\eta>0$ and almost surely finite $k_2$ such that
\[
\sum_{k=k_1}^{\ell} \frac{(\mathbbm{1}(C_k \cap G_k \cap B_k \setminus \cup_{i=k_1}^{k-1} A_i)}{1+ \log \psi^{-1}(a(k)) } \geq \eta \sum_{k=k_1}^{\ell} \frac{1}{1+ \log \psi^{-1}(a(k)) }
\]
for every $\ell\geq k_2$. Since $a(k)-a(k-1)\to\infty$ as $k\to\infty$, the other term in \eqref{eq:nearly_there} is of lower order than this and we deduce that
\[\sum_{k=k_1}^\infty \mathbf{P}(A_k \setminus \cup_{i=k_1}^{k-1} A_i)
=\infty.
\]
It follows from the Borel-Cantelli counterpart that $\mathbf{P}(\cup_{k=k_1}^\infty A_k)=1$, contradicting the definition of $k_1$. \qedhere
\end{proof}

\subsection{Completing the proof of Theorem \ref{thm:mainPolyDecay}}

In this section we deduce Theorem \ref{thm:mainPolyDecay} from Theorem \ref{thm:kcuts}. Note that the proof establishes a slightly stronger claim giving the almost sure existence of infinitely many cut times on the event that hitting probabilities decay quickly, without needing to assume that the latter occurs almost surely.
\begin{proof}[Proof of Theorem \ref{thm:mainPolyDecay}]
Let $M=(\Omega,P,\dagger)$ be a transient Markov chain with killing, let $X=(X_n)_{n\geq 0}$ be a trajectory of  $M$, and let $\phi:[0,\infty)\rightarrow[0,\infty)$ be a strictly increasing function  such that
	\begin{equation}\label{eq:general_condition}
	\sum_{n=1}^\infty \frac{1}{1 \vee \log\left(\phi^{-1}(n)\right)}=\infty.
	\end{equation} 
	It suffices by Lemma \ref{lem:integral} to prove that if the event
	$\mathscr{G}=\{ \limsup_{n\to\infty} e^{\phi(n)} \bbH(X_n,X_m) < \infty$ for every $m\geq 0$ such that $X_m \neq \dagger\}$
	has positive probability, then $X$ is either killed or has infinitely many cut times almost surely conditional on $\mathscr{G}$. Compared to this statement, Theorem \ref{thm:kcuts} has three additional hypotheses: that $M$ is locally finite, that $M$ is irreducible, and that $P(x,x)=0$ for every $x\neq \dagger$. We will show that these assumptions can each be removed via a simple reduction argument.

\medskip

\noindent \textbf{Removing the condition that $P(x,x)=0$ for every $x\neq \dagger$:} First suppose that $M$ is irreducible and locally finite but does not necessarily satisfy $P(x,x)=0$ for every $x\neq \dagger$. Since $M$ is irreducible and transient, $P(x,x) \neq 1$ for every $x \neq \dagger$. Consider the Markov chain $M'=(\Omega,P',\dagger)$ where $P'$ is the transition matrix defined by $P'(x,x)=0$ for every $x\neq \dagger$ and
\[
P'(x,y) = \frac{P(x,y)}{1-P(x,x)} \qquad \text{ for every $x \in \Omega \setminus \{\dagger\}$ and $y\in \Omega \setminus \{x\}$.}
\]
 We can couple trajectories $X$ and $Y$ of $M$ and $M'$ so that $X$ visits the same states as $Y$ in the same order but possibly includes additional steps where it stays at the same non-graveyard vertex for more than one consecutive step. In particular, if $Y$ has infinitely many cut times then $X$ does also. Since the hitting probabilities for $M$ and $M'$ are equal and $Y_n \in \{X_m : m \geq n\}$ for every $n\geq 1$, if the event $\mathscr{G}$ holds for $X$ then the analogous event holds for $Y$ also, and the claim follows from Theorem~\ref{thm:kcuts}.

\medskip

\noindent \textbf{Removing the condition that $M$ is irreducible:} Now suppose that $M$ is locally finite but not necessarily irreducible, and does not necessarily satisfy $P(x,x)=0$ for every $x\neq \dagger$. 
 For each communicating class $C\neq \{\dagger\}$ of $M$ we can define a Markov chain with killing $M_C=(C \cup \{\dagger\},P_C,\dagger)$, where $P_C(u,v) = P(u,v)$ for each $u,v \in C$ and $P_C(u,\dagger) = \sum_{v\notin C} P(u,v)$ for each $u\in C$.
 When a trajectory $X$ of the original Markov chain $M$ enters a communicating class $C \neq \{\dagger\}$, it can be coupled with a trajectory of $M_C$ up to the first time that it leaves $C$, at which time the coupled trajectory of $M_C$ is killed. 
 Observe that a trajectory of $M$ must either pass though infinitely many communicating classes or enter some final communicating class $C_f$. If $C_f = \{\dagger\}$, the trajectory is killed and there is nothing to prove.
 Each time the trajectory $(X_n)$ enters a new communicating class $C \neq \{\dagger\}$, the coupling with a trajectory of $M_C$ together with the previous part of the proof implies that, conditional on $\mathscr{G}$,  the walk will almost surely either stay in $C$ forever and have infinitely many cut times or leave $C$.
 Thus, if $\mathscr{G}$ holds and $X$ eventually stays in a single communicating class, then it is either killed or has infinitely many cut times almost surely. On the other hand, if 
 $X$ visits infinitely many communicating classes then the set of times at which it enters a new communicating class constitute an infinite set of cut times, so that the claim also holds in this case.
 
\medskip

 \noindent
 \textbf{Removing the condition that $M$ is locally finite:} We now let $M$ be arbitrary; it remains only to remove the restriction that it is locally finite.
   We assume that trajectory $X$ starts at a non-recurrent state $X_0\in\Omega$, the claim holding vacuously otherwise. We merge all the recurrent communicating classes of $M$ into the single state $\dagger$ to give a Markov chain with killing $M^\prime=(\Omega^\prime,P^\prime,\dagger)$, noting that we can couple trajectories of $M$ and $M^\prime$ such that they are identical up to the first time the two trajectories enter a recurrent communicating class (which corresponds to be killed in $M'$). We enumerate the states in $\Omega^\prime\setminus\{\dagger\}$ as $(y_i)_{i\geq 1}$ and for each state $y\in \Omega$ define $y^\rightarrow=\{z\in\Omega:P(y,z)>0\}$. Fix $\varepsilon>0$. Since every state in $\Omega'\setminus \{\dagger\}$ is transient,  we can select for each $i\geq 0$ a subset $L_i$ of the states in $y_i^\rightarrow$ such that $y_i^\rightarrow\setminus L_i$ is finite and the trajectory $(X_n)$ on $M^\prime$ starting at $X_0$ satisfies
 \[
 \pr(\exists j\in\N \text{ such that } X_j=y_i \text{ and } X_{j+1}\in L_i)<\varepsilon 2^{-i}.
 \]
It follows by a union bound that the event $\mathscr{L}=\{\exists i,j\in\N$ such that $X_j=y_i$ and $X_{j+1}\in L_i\}$ that the trajectory ever makes a transition of this type has probability at most $\varepsilon$.
We construct a new Markov chain with killing $M^{\prime\prime}=(\Omega^\prime,P^{\prime\prime},\dagger)$ where, for each $i\geq 1$, transitions from $y_i$ to $L_i$ are redirected to the graveyard state. That is, for each $i\geq 1$, we set $P^{\prime\prime}(y_i,v)=0$ for every $v\in L_i$, set $P''(y_i,v)=P'(y_i,v)$ for each $v \notin L_i \cup \{\dagger\}$, and set $P^{\prime\prime}(y_i,\dagger)=P^{\prime}(y_i,\dagger)+\sum_{v\in L_i} P^{\prime}(y_i,v)$. This construction ensures that $M''$ is locally finite.
 We can couple trajectories $X$ on $M^\prime$ and $Y$ on $M^{\prime\prime}$ to be identical up until the time that $X$ makes a transition from $y_i$ to $L_i$ for some $i\geq 1$, after which $Y$ is killed.
  It follows from this coupling that
   \begin{equation}\label{eq:greendecrease}
	\mathbb{H}^M(x,y)\geq\mathbb{H}^{M^{\prime\prime}}(x,y) \qquad \text{ for every $x,y\in \Omega' \setminus \{\dagger\}$},
\end{equation}
and hence under this coupling that $\mathbb{H}^{M^{\prime\prime}}(Y_n,Y_m)\leq\mathbb{H}^{M^{\prime\prime}}(X_n,X_m)$ whenever $n\geq m$ is such that $Y_n \neq \dagger$.
Since $M''$ is locally finite it follows that, under this coupling, $Y$ is either killed or has infinitely many cut times on the event  $\mathscr{G}=\{ \limsup_{n\to\infty} e^{\phi(n)} \bbH(X_n,X_m) < \infty$ for every $m\geq 0$ such that $X_m \neq \dagger\}$. The claim follows since $X$ and $Y$ coincide forever with probability at least $1-\varepsilon$ and $\varepsilon>0$ was arbitrary.
\end{proof}

\section{Superdiffusive walks} \label{sec:SSD}

In this section we prove \cref{thm:speedTheorem}, which states that random walks on networks satisfying a mild superdiffusivity condition have infinitely many cut times almost surely. It will once again be convenient to work within a more general framework that allows for random walks to be killed.
 We define a \textbf{network with killing} to be a tuple $N=(V,E,c,K)$ where $(V,E,c)$ is a network and $K:V\to [0,\infty)$ is a \textbf{killing function}. Given a network with killing $N=(V,E,c,K)$, the random walk on $N$ is the Markov chain with state space $V \cup \{\dagger\}$ and with transition matrix defined by
 \[
P(u,v) = \frac{c(u,v)}{c(u)+K(u)} \quad \text{for $u,v\in V$} \qquad P(u,\dagger) = \frac{K(u)}{c(u)+K(u)} \quad \text{for $u\in V$}, \quad \text{ and } \quad P(\dagger,\dagger)=1,
 \]
 where $c(u)$ denotes the total conductance of all oriented edges emanating from $u$ and $c(u,v)$ denotes the total conductance of all oriented edges with tail $u$ and head $v$.
We will follow the standard practice of writing $p_n(u,v)=P^n(u,v)$ for transition probabilities.

\medskip

The starting point of our analysis is the following well-known theorem of Varopoulos and Carne \cite{MR822826,MR837740} (see also \cite{MR3616205}). While usually stated without allowing for killing, the same proof applies equally well to networks with killing; the important thing is that $P$ satisfies the self-adjointness relation $(c(u)+K(u)) P(u,v) = (c(v)+K(v))P(v,u)$ for every $u,v\in V$ and that the restriction of $P$ to $V$ is substochastic.

\begin{theorem}[Varopoulos-Carne Inequality]
    The transition probabilities $p_n(x,y)$ of a simple random walk on a network with killing $N=(V,E,c,K)$ satisfy
    \begin{equation} \label{eq:VC}
        p_n(x,y)\leq \sqrt{\frac{c(y)+K(y)}{c(x)+K(x)}} \exp\left[-\frac{d(x,y)^2}{2n}\right] \rho^n.
    \end{equation}
for every $x,y\in V$ and $n\geq 1$,    where $\rho$ is the spectral radius of the restriction of $P$ to $V$.
\end{theorem}

We will bound the spectral radius term trivially by $1$ in all our applications of this inequality.

\medskip

While we would naively like to use Varoupoulos-Carne together with our superdiffusivity hypothesis to obtain bounds on the decay of the Green's function along the random walk, and conclude by applying \cref{thm:mainPolyDecay}, unfortunately, this
 does not seem to be possible in general. Indeed, while it is possible to obtain bounds on the small-time and medium-time transition probabilities of the walk using the Varopoulos-Carne inequality, this inequality gives us no control of the large-time contributions to the Green's function.
In our efforts to circumvent this issue, we will establish some rather general conditions under which we can compare the decay of $\mathbb{G}(X_n,X_0)$ and $p_n(X_n,X_0)$ that may be of independent interest.

\subsection{Comparing $p_n(X_n,X_0)$ and $\mathbb{G}(X_n,X_0)$ assuming superpolynomial decay}

The first step of our proof is to give conditions under which the a.s.\ rates of decay of $p_n(X_n,X_0)$ and $\mathbb{G}(X_n,X_0)$ can be compared. Given a connected network with killing $N$, we say that $N$ satisfies the \textbf{superpolynomial decay condition} if
\begin{equation}
\label{eq:SPD}
\tag{SPD}
\lim_{n\to\infty} \frac{\log \sup_{u\in V} p_n(u,v)}{\log n} =0 \qquad \text{ for some (and hence every) $v\in V$.}
\end{equation}

\begin{proposition}
\label{prop:p_and_g}
Let $N$ be a network with killing and let $X=(X_n)_{n\geq 0}$ be a random walk started at $o$. If the transition probabilities to $o$ satisfy the superpolynomial decay condition \eqref{eq:SPD} then
\[
\lim_{n\to\infty}\frac{\log p_n(X_n,X_0)}{\log \mathbb{G}(X_n,X_0)} = 1
\]
almost surely, with the convention that this ratio is equal to $1$ when $X_n=\dagger$.
\end{proposition}

This proposition is not really needed for \cref{thm:speedTheorem}, since the superpolynomial decay hypothesis \eqref{eq:SPD} would already suffice to deduce the claim from \cref{thm:mainPolyDecay}. It will, however, be used more seriously in the proof of \cref{thm:speedTheorem_density}.
For random walks on finitely generated groups with positive speed, which always satisfy \eqref{eq:SPD} by \cite[Corollary 6.32]{MR3616205},  \Cref{prop:p_and_g} implies that the Av\'ez entropy and exponential decay rate of the Green's function coincide, recovering a result of Benjamini and Peres \cite[Proposition 6.2]{MR1254826}. Similar results for groups that are \emph{not} finitely generated have been obtained in \cite{MR2408585}.

\medskip

\Cref{prop:p_and_g} will be deduced from the following elementary observation.

\begin{lemma}  \label{lem:ratios}
    The transition probabilities $p_n(x,y)$ of a simple random walk $(X_n)_{n\geq 0}$ on a network with killing $N=(V,E,c,K)$ satisfy
    \begin{equation} \label{eq:ratios}
        \E{\frac{p_m(X_n,X_0)}{p_n(X_n,X_0)}}\leq \mathbb{P}(X_m \neq \dagger) + \mathbb{P}(X_n = \dagger)\leq 2,
    \end{equation}
    for every $x\in V$ and $m,n\geq 0$, with the convention that the ratio is $1$ when $X_n=\dagger$.
\end{lemma}
\begin{proof}[Proof of \cref{lem:ratios}]
    Let  $A$ be the set of vertices $x\in V$ such that $p_n(X_0,x)>0$. Then we have that \begin{multline*}
    \E{\frac{p_m(X_n,X_0)}{p_n(X_n,X_0)} \mathbbm{1}(X_n \neq \dagger)} = \E{\frac{p_m(X_0,X_n)}{p_n(X_0,X_n)} \mathbbm{1}(X_n \neq \dagger)}   \\=  \sum_{x\in A} p_n(X_0,x) \frac{p_m(X_0,x)}{p_n(X_0,x)}= \sum_{x\in A} p_m(X_0,x)\leq \mathbb{P}(X_m \neq \dagger),
    \end{multline*}
    which is easily seen to imply the claim.
\end{proof}

\begin{proof}[Proof of \Cref{prop:p_and_g}]
Using \cref{lem:ratios}, an application of the Borel-Cantelli Lemma implies that there exists an almost surely finite  random variable $\gamma$
such that 
\begin{equation}
\label{eq:m2n2}
p_m(X_n,X_0)\leq \gamma (m+1)^2(n+1)^2 p_n(X_n,X_0)
\end{equation}
for every $n,m\geq0$. Fix $\varepsilon>0$, let $n \geq 1$ and let $N= \lceil p_n(X_n,X_0)^{-\varepsilon}\rceil$. We deduce by summing \eqref{eq:m2n2} over $0\leq m \leq N$ that
\[
\mathbb{G}(X_n,X_0) = \sum_{m=0}^\infty p_m(X_n,X_0) \leq \gamma (N+1)^3 (n+1)^2 p_n(X_n,X_0) + \sum_{m=N+1}^\infty p_m(X_n,X_0).
\]
Since $p_m(X_n,X_0) \leq \sup_v p_m(v,X_0)$ decays superpolynomially in $m$ by \eqref{eq:SPD} we can write this estimate in asymptotic notation as 
\[
\mathbb{G}(X_n,X_0) \leq p_n(X_n,X_0)^{1-3\varepsilon-o(1)}+p_n(X_n,X_0)^{\omega(1)} \qquad \text{ a.s.\ as $n\to\infty$ for each fixed $\varepsilon>0$,}
\]
where $o(1)$ and $\omega(1)$ denote quantities tending to $0$ and $+\infty$ respectively. The claim follows since $\varepsilon>0$ was arbitrary and the inequality $\mathbb{G}(X_n,X_0)\geq p_n(X_n,X_0)$ holds trivially.
\end{proof}

Since $p_n(X_n,X_0)$ decays superpolynomially under the superdiffusivity assumption \eqref{eq:superdiffusivity} by Varopoulos-Carne and we only require polynomial decay of $\mathbb{G}(X_n,X_0)$ to apply \cref{thm:mainPolyDecay}, to prove \cref{thm:speedTheorem} it would suffice for us to have a much weaker comparison of the two quantities than that provided by \Cref{prop:p_and_g}. Such comparison inequalities can be provided by the proof of \Cref{prop:p_and_g} under much weaker assumptions on transition probabilities that are only barely stronger than transience. For example, this argument is able to handle the ballistic case under the mild additional assumption that there exists $c>0$ such that
\begin{equation}
\label{eq:logp_ncondition}
\sup_u p_n(u,v) \leq \frac{C_v}{n (\log n)^{1+c}} \qquad \text{ for every $v \in V$ and $n\geq 2$},
\end{equation}
where $C_v$ is a finite constant depending on the choice of $v$.
Unfortunately, we believe that such transition probability estimates need not hold in general, even when the random walk has positive speed. Indeed, identifying the origin of $\Z^2$ with the root of a binary tree gives an example where the random walk has positive liminf speed almost surely but where $p_n(0,0)$ is at least the probability that the walk makes an excursion of length $n$ from the origin to itself in $\Z^2$, which is of order $n^{-1} (\log n)^{-2}$. Replacing $\Z^2$ in this example by a tree of slightly superquadratic growth should allow one to construct examples where the random walk has positive speed but where \eqref{eq:logp_ncondition} does not hold for any $c>0$; we do not pursue this further here. We believe that there exist examples where the random walk has positive speed but where $\mathbb{G}(X_n,X_0)$ decays very slowly, but this seems to require a more involved construction.

\subsection{Spatially-dependent killing and the proof of Theorem \ref{thm:speedTheorem}}

We now describe how we circumvent the issue discussed at the end of the previous subsection by introducing \emph{spatially dependent killing} to our network, where we will take $K(x)$ to be a function of the distance of $x$ from some fixed origin vertex $o$. We will show under the hypotheses of \cref{thm:speedTheorem} that this killing function can be chosen to decay sufficiently quickly that the random walk has a positive probability never to be killed, but decay sufficiently slowly that the resulting network with killing satisfies \eqref{eq:SPD}.

\medskip

We begin by finding the marginal rate of decay under which the resulting network with killing automatically satisfies \eqref{eq:SPD}. Given a network $N=(V,E,c)$ and a fixed origin vertex $o$, we write $\langle x \rangle = 2 \vee d(o,x)$ for each $x\in V$ to avoid division by zero.

\begin{lemma} \label{lem:combbound} 
Let $N=(V,E,c)$ be a network with $c_{\min}=\inf_{x\in V} c(x)>0$, fix a vertex $o\in V$, let $\gamma \in \R$ and let $K:V\to [0,\infty)$ be the killing function defined by $K(x) = c(x) \min\{1,$ $\langle x \rangle^{-2} (\log \langle x \rangle)^\gamma \}$. Then there exists a positive constant $c=c(\gamma)$ such that
    \[  
    p_n(x,o) \leq \sqrt{\frac{8c(o)}{c_{\min}}} \exp\Big[-c\, (\log n)^{\gamma/2}\Big]
    \]
    for every $x\in V$ and $n\geq 2$. In particular, if $\gamma>2$ then $(V,E,c,K)$ satisfies \eqref{eq:SPD}.
\end{lemma}

The rough idea behind this lemma is as follows: Suppose we run a random walk for  time $n$ started at some vertex $x$. If $d(o,X_m) \gg \sqrt{n}$ for some $0\leq m \leq n$ then the probability of hitting the origin at time $n$ is small as a consequence of Varopoulos-Carne. On the other hand, if this never happens, the higher rate of killing ensures that the walk is killed before time $n$ with high probability and is therefore unlikely to hit the origin at time $n$.

\begin{proof}[Proof of Lemma~\ref{lem:combbound}]
Let $\mathbb{P}_x$ denote the law of the random walk $X=(X_n)_{n\geq 0}$ on the network with killing $(V,E,c,K)$ started at some fixed vertex $x \in V$, and let $\tau_\dagger$ denote the time the walk is killed (i.e.\ first visits the graveyard state $\dagger$). We define $d(o,\dagger)=\infty$ and decompose
\begin{multline}
\label{eq:split}
\mathbb{P}_x(X_n=o) = 
\mathbb{P}_x(X_n=o \text{ and } d(o,X_m) > r \text{ for some } 0\leq m \leq n)
\\+
\mathbb{P}_x(X_n=o \text{ and } d(o,X_m) \leq r \text{ for every $0\leq m \leq n$})
\end{multline}
for  each $n,r \geq 2$, where $r$ is a parameter we will optimize over at the end of the proof. 
    We begin by analysing the first term on the right hand side of \eqref{eq:split}. Let $\kappa$ be the stopping time $\kappa:=\inf \{m \geq 0: d(o,X_m) >r\}$. We apply the strong Markov property at $\kappa$ together with Varopoulos-Carne to give that
    \begin{align*}
    \mathbb{P}_x(X_n=o \text{ and } d(o,X_m) > r \text{ for some } 0\leq m \leq n) 
    &\leq \sum_{m=0}^n \sum_{z \in V} \mathbb{P}_x(\kappa=m, X_\kappa =z)\mathbb{P}_z(X_{n-m}=o)\nonumber\\
    &\leq  \sqrt{\frac{c(o)+K(o)}{c(z)+K(z)}} \exp\left[ - \frac{r^2}{2n} \right] \leq \sqrt{\frac{2c(o)}{c_{\min}}} \exp\left[ - \frac{r^2}{2n} \right],
    \end{align*}
    where $c_{\min}=\inf_{z\in V} c(z)$, and where the final inequality follows by definition of $K$.
We now turn our attention to the second term on the right hand side of \eqref{eq:split}. Each time the walk makes a step at distance at most $r$ it is killed with probability at least $\frac{1}{2} (1 \wedge r^{-2} (\log r)^\gamma)$. Letting $c_1=c_2(\gamma)$ be a positive constant such that this probability is at least $c_1 r^{-2} (\log r)^\gamma$, we deduce that
    \begin{align*}
    \mathbb{P}_x\bigl(X_n=o \text{ and } d(o,X_m) \leq r \text{ for every $0\leq m \leq n$}\bigr)
     &\leq
    \mathbb{P}_x\bigl(\tau_\dagger > n \text{ and } d(o,X_m) \leq r \text{ for every $0\leq m \leq n$}\bigr)\\ &\leq  \left(1-\frac{c_1 (\log r)^\gamma}{r^2}\right)^n \leq \exp\left[- \frac{c_1(\log r)^\gamma n}{r^2}\right],
    \end{align*}
    where we used the inequality $1-t \leq e^{-t}$ in the final inequality. Substituting these two estimates into \eqref{eq:split} yields that
    \[
\mathbb{P}_x(X_n=o) \leq \sqrt{\frac{2 c(o)}{c_{\min}}} \left(\exp\left[ - \frac{r^2}{2n} \right] + \exp\left[- \frac{c_1(\log r)^\gamma n}{r^2}\right]\right),
    \] 
    and the claim follows by taking $r= \lceil n^{1/2} (\log n)^{\gamma/4} \rceil$. \qedhere
    
\end{proof}

Let $N=(V,E,c)$ be a network, let $o$ be a vertex of $N$, and let $X=(X_n)_{n\geq 0}$ be the random walk on $N$. Let $r>0$ and let $\mathscr{S}_r$ be the event that
\[
\liminf_{n\to\infty} \frac{d(o,X_n)}{n^{1/2} (\log n)^r} >0.
\]
We next wish to show that for any choice of $r$, we can choose the killing function $K$ as in \cref{lem:combbound} such that if $\mathscr{S}_r$ holds, the walk does not ``feel" the effects of the killing. More precisely, we can ensure the killing function decays quickly enough such that conditional on the path of the walk, the walk almost surely has a positive probability of never getting killed. To formulate this lemma, let us first note that we can couple the random walks on $(V,E,c)$ and $(V,E,c,K)$ so that they coincide up until the killing time $\tau_\dagger$. Writing $X$ for the unkilled walk and writing $\mathbf{P}_x$ for the joint law of $X$ and $\tau_\dagger$ when $X$ is started at $x\in V$, this coupling is determined by the equality
\[
\mathbf{P}_x(\tau_\dagger = n \mid X) = K(X_{n-1}) \prod_{i=0}^{n-2} (1-K(X_i)).
\]

\begin{lemma} \label{lem:feels} 
Let $N=(V,E,c)$ be a network with $c_{\min}=\inf_{x\in V} c(x)>0$, fix a vertex $o\in V$, let $\gamma \in \R$, and let $K:V\to [0,\infty)$ be the killing function defined by $K(x) = c(x) \min\{1,$ $\langle x \rangle^{-2} (\log \langle x \rangle)^\gamma \}$. If $X$ is a random walk on $N$ and $\gamma+1 < 2r$, then
$\mathbf{P}_x(\tau_\dagger = \infty \mid X) > 0$ 
almost surely on the event $\mathscr{S}_r$.
\end{lemma}

\begin{proof}[Proof of \cref{lem:feels}]
We can write the conditional probability $\mathbf{P}_x(\tau_\dagger = \infty \mid X)$ as an infinite product
\[
\mathbf{P}_x(\tau_\dagger = \infty \mid X) = \prod_{i=0}^\infty (1-K(X_i)), \quad \text{ which is positive if and only if } \quad \sum_{i=0}^\infty K(X_i) <\infty.
\]
We have by calculus that there exists a random variable $\alpha$ taking values in $[1,\infty]$ that is finite on the event $\mathscr{S}_r$ and satisfies
$K(X_n) \leq \alpha (\log n)^{\gamma - 2r} n^{-1}$
for every $n\geq 1$, and it follows that if $2r>1+\gamma$ then $\sum_{i=0}^\infty K(X_i)<\infty$ on the event $\mathscr{S}_r$ as required.
\end{proof}

We are now ready to complete the proofs of \cref{thm:speedTheorem,thm:speedTheorem_density}.

\begin{proof}[Proof of \cref{thm:speedTheorem}]
Let $r>3/2$ and $2<\gamma<2r-1$.
Let $N=(V,E,c)$ be a network with $c_{\min}=\inf_{x\in V} c(x)>0$, fix a vertex $o\in V$, and let $K:V\to [0,\infty)$ be the killing function defined by $K(x) = c(x) \min\{1,$ $\langle x \rangle^{-2} (\log \langle x \rangle)^\gamma \}$. Couple the random walk $X$ on $N$ with the killing time $\tau_\dagger$ as above, write $X^\dagger$ for the killed walk, and assume that the superdiffusivity event $\mathscr{S}_r$ has positive probability.
Let $p^\dagger_n$ and $\mathbb{G}_\dagger$ denote transition probabilities and the Green's function with respect to the killed network $N_\dagger = (V,E,c,K)$. \cref{lem:combbound} implies that $N_\dagger$ satisfies the superpolynomial decay condition \eqref{eq:SPD}, and we deduce from \Cref{prop:p_and_g} that 
\begin{equation}
\label{eq:daggerlogratio}
\lim_{n\to\infty}\frac{\log p^\dagger_n(X_n^\dagger,X_0^\dagger)}{\log \mathbb{G}_\dagger(X_n^\dagger,X_0^\dagger)}=1
\end{equation}
almost surely, where the ratio is considered to be equal to $1$ when $X_n^\dagger=\dagger$. Varopoulos-Carne yields that
\begin{align*}
p^\dagger_n(X_n^\dagger,X_0^\dagger) &= \exp\left[-\Omega\left((\log n)^{2r}\right) \right] & \text{ as $n\to\infty$}
\intertext{when $\mathscr{S}_r$ holds, and hence by \eqref{eq:daggerlogratio} that}
\mathbb{G}_\dagger(X_n^\dagger,X_0^\dagger) &= \exp\left[-\Omega\left((\log n)^{2r}\right) \right] & \text{ as $n\to\infty$}
\end{align*}
almost surely on the event $\mathscr{S}_r$. (Here we recall that $\Omega(f(n))$ denotes a quantity that is lower bonded by a (possibly random) positive multiple of $f(n)$ for large values of $n$.) Since $r>3/2> 1/2$, this decay is superpolynomial, and it follows from \cref{thm:mainPolyDecay} that $X^\dagger$ is either killed or has infinitely many cut times almost surely on the event $\mathscr{S}_r$. Since we also have that the conditional probability $\mathbf{P}_x(\tau_\dagger = \infty \mid X)$ is almost surely positive on the event $\mathscr{S}_r$, we deduce that $X$ has infinitely many cut times almost surely on the event $\mathscr{S}_r$ as claimed.
\end{proof}

\begin{proof}[Proof of \cref{thm:speedTheorem_density}]
Since $G$ has bounded degrees and the walk has positive liminf speed almost surely, it follows as above that we can take a bounded killing function $K$ so that the walk has a.s.\ positive conditional probability not to be killed and the killed Green's function $\mathbb{G}_\dagger(X_n^\dagger,X_o^\dagger)$ decays exponentially. On the other hand, since the degrees and the killing function are both bounded, there exists a positive constant $c$ such that $\mathbb{G}_\dagger(X_{n+1}^\dagger,X_o^\dagger) \geq c \cdot \mathbb{G}_\dagger(X_{n}^\dagger,X_o^\dagger)$ for every $n\geq 0$ such that $X_{n+1}^\dagger \neq \dagger$. Combined with exponential decay this implies that if we define $A_a=\{n : Z_{n+1}^\dagger \leq a \min_{m\leq n}Z_m^\dagger\}$, where $Z_n^\dagger=\mathbb{G}_\dagger(X_{n}^\dagger,X_o^\dagger)$, and define $\mathcal{A}_a$ to be the event that $A_a$ has positive $\liminf$ density then $\bigcup_{k=1}^\infty \mathcal{A}_{(k-1)/k}$ has probability $1$. 
On the other hand, by optional stopping, for each $n \geq 1$ the conditional probability that $n$ is a cut time given everything the walk has done up to time $n$ is bounded below by $1-a$ whenever $n\in A_a$. From here the claim follows easily by standard arguments and we omit the details.
\end{proof}


\section{Sharpness for birth-death chains} \label{sec:sharp}
In this final section, we demonstrate that the integral condition given in Theorem \ref{thm:mainPolyDecay} is sharp by comparing our results to those of Cs\'aki, F\"oldes, and R\'ev\'esz \cite{MR2644879} on the cut times of birth-death chains. Throughout this section, $(X_n)_{n\geq0}$ will denote a random walk on $\Z_{\geq 0}$ with transition probabilities of the form \begin{equation*}
	E_i:=\pr(X_{n+1}=i+1\mid X_n=i)=1-\pr(X_{n+1}=i-1\mid X_n=i)=\begin{cases}
		1 &\text{if } i=0\\
		1/2+p_i&\text{otherwise}
	\end{cases},
\end{equation*}
where $-1/2< p_i<1/2$ for each $i\geq 1$.
For each $m\geq 0$, define \[
D(m) = 1+\sum_{j=1}^{\infty} \prod_{i=1}^{j} \bigg(\frac{1}{E_{m+i}}-1\bigg).
\]
The aforementioned work \cite{MR2644879} establishes the following dichotomy. (Here we rephrase their theorem in terms of cut times and omit the strengthened conclusion concerning strong cut points.)
\begin{theorem}[{\cite[Theorem 1.1]{MR2644879}}]\label{thm:import}
	Let $(X_n)_{n\geq 0}$ be a transient birth-death chain as defined above with $0\leq p_i<1/2$ for each $i\geq 1$.
	\begin{itemize}
		\item If 
		$\sum_{n=2}^\infty (D(n)\log n)^{-1}<\infty$,
		then $(X_n)$ has finitely many cut times a.s.
		\item If $D(n)\leq n(\log n)^{1/2}$ and $\sum_{n=2}^\infty (D(n)\log n)^{-1}=\infty$,
		then $(X_n)$ has infinitely many cut times a.s.
	\end{itemize}
\end{theorem}
We use this Theorem to prove the following partial converse of Theorem \ref{thm:mainPolyDecay}.  
We let $\mathbb{G}(n)=\mathbb{G}(n,0)$ denote the Green's function associated with $(X_n)$ and say a function $F:[0,\infty)\rightarrow(0,\infty)$ is \textbf{eventually log-convex} if there exists $r\geq 0$ such that the restriction of $F$ to the interval $[r,\infty)$ is log-convex.
\begin{proposition} \label{prop:sharpness}
	Given any decreasing differentiable bijection $\Phi:[0,\infty)\rightarrow(0,1]$ that is eventually log-convex and satisfies
	\[\int_0^{1} \frac{1}{u (1\vee \log \Phi^{-1}(u))}\ \mathrm{d}u<\infty,\]
	there exists a nearest-neighbour random walk $(X_n)_{n\geq 0}$ on $\Z_{\geq 0}$ with Green's function $\mathbb{G}(n)=\mathbb{G}(n,0)$, such that
		$\limsup_{n\to\infty}\Phi(n)^{-1}\mathbb{G}(X_n) < \infty$  
	and $(X_n)_{n\geq 0}$ has at most finitely many cut times almost surely.
\end{proposition}
In the proof of this proposition, we will utilize the following two elementary identities relating the quantities $p_n$, $\mathbb{G}(n)$, and $D(n)$ for each $n\geq 1$:
\begin{center}
	\begin{minipage}[b]{.45\textwidth}
		\vspace{-\baselineskip}
		\begin{equation}
			D(n-1)= \frac{\mathbb{G}(n-1)}{\mathbb{G}(n-1)-\mathbb{G}(n)} \label{eq:DGrel}
		\end{equation}
	\end{minipage}
	\hfill and\hfill
	\begin{minipage}[b]{.45\textwidth}
		\vspace{-\baselineskip}
		\begin{equation}
			p_n = \frac{1}{2}\frac{\mathbb{G}(n-1)+\mathbb{G}(n+1)-2\mathbb{G}(n)}{\mathbb{G}(n-1)-\mathbb{G}(n+1)}. \label{eq:pGrel}
		\end{equation}
	\end{minipage}
\end{center}
The first identity follows from \cite[(2.1)]{MR2644879} and the elementary identity $\mathbb{H}(n+1,n)=\mathbb{H}(n+1)/\mathbb{H}(n)$, where $\mathbb{H}(n)$ is the probability that $(X_m)$ will hit $0$ when $X_0=n$, and the second identity follows from \cite[(2.2)]{MR2644879} together with the first.

\medskip

Plugging \eqref{eq:DGrel} into $\sum_{n=2}^\infty (D(n)\log n)^{-1}$, we observe that their summation criterion is roughly related to our integral condition by a change of variables. We prove Proposition \ref{prop:sharpness} by formalising this relationship for a walk whose Green's function is an appropriate transformation of the input function $\Phi$.  We then conclude by proving a very weak lower bound on the displacement of the walk from $0$.
\begin{proof}[Proof of Proposition \ref{prop:sharpness}]
	Let $f$ be the decreasing, log-convex function $f(x):=e^{-\sqrt{\log(x+2)}}$, and let $M\geq 2$ be the smallest integer such that the restriction of $\Phi$ to $[M,\infty)$ is log-convex.	We begin by defining the function $\widetilde{\Phi}:[0,\infty)\rightarrow(0,\infty)$ by
	\[
	\widetilde{\Phi}(x)=\Phi((x+M)^4)f(x),
	\]
	and noting some its properties. First, observe that $\widetilde{\Phi}\leq\Phi$ is strictly positive, strictly decreasing, log-convex and differentiable. Moreover, since $\widetilde{\Phi}(x)\leq\Phi((x+M)^4)\wedge f(x)$, we also have that $\widetilde{\Phi}^{-1}(x)\geq(\Phi^{-1}(x)^{1/4}-M)\vee f^{-1}(x)$, and hence that there exists a $C<\infty$ finite such that
	\begin{align} \label{eq:stillfinite}
		\int_0^{1} \frac{1}{u (1\vee \log \widetilde{\Phi}^{-1}(u))}\ \mathrm{d}u
        &\leq\int_0^{1} \frac{1}{u (1\vee \log [(\Phi^{-1}(x)^{1/4}-M)\vee f^{-1}(x)])}\ \mathrm{d}u\nonumber \\
		 &= \int_0^{1}\min\left\{ \frac{1}{u (1\vee \log f^{-1}(x))},\,  \frac{1}{u (1\vee \log [(\Phi^{-1}(x)^{1/4}-M)])}\right
	\}\  \mathrm{d}u \nonumber\\
	 &\leq C+C	\int_0^{1} \frac{1}{u (1\vee \log \Phi^{-1}(x))} \mathrm{d}u<\infty,
	\end{align}
	where for functions $F\in\{\widetilde{\Phi},f\}$, we use the convention that $1 \vee \log F^{-1}(u) =1$ when $F^{-1}(u)$ is not defined.
    We also note that the logarithmic derivative $(\log\widetilde{\Phi})^\prime$ of $\widetilde{\Phi}$, which is increasing by log-convexity of $\widetilde{\Phi}$, satisfies the inequality
	\begin{equation} \label{eq:loweblog}
		-\frac{\mathrm{d}}{\mathrm{d}x} \log \widetilde{\Phi}(x) \geq -\frac{\mathrm{d}}{\mathrm{d}x} \log f(x) =
        \frac{1}{2(x+2)\sqrt{\log (x+2)}}.
	\end{equation}
	We now use the function $\widetilde{\Phi}$ to define a Markov chain satisfying the desired properties. For $i\geq 1$, we define 
	\begin{equation}
    \label{eq:pphi_def}
	p_i = \frac{1}{2} \frac{\widetilde{\Phi}(n-1)+\widetilde{\Phi}(n+1)-2\widetilde{\Phi}(n)}{\widetilde{\Phi}(n-1)-\widetilde{\Phi}(n+1)},
	\end{equation}
	which is non-negative since $\widetilde \Phi$ is convex and strictly less than $1/2$ since $\widetilde{\Phi}$ is strictly decreasing.
     We can therefore define a nearest-neighbour random walk $(X_n)_{n\geq 0}$ on the integers with $X_0=0$ and with transition probabilities
	\[
	\pr(X_{n+1}=i+1\mid X_n=i)=1-\pr(X_{n+1}=i-1\mid X_n=i)=\frac{1}{2}+p_i\qquad \text{for } i\geq1,
	\] and $\pr(X_{n+1}=1\mid X_n=0)=1$.
	Comparing \eqref{eq:pphi_def} and \eqref{eq:pGrel}, it follows by induction on $n$ that the Green's function of this Markov chain is given by
	\[\mathbb{G}(n)=C\widetilde{\Phi}(n)\qquad \text{for }n\geq 0,\]
	for some constant $C=\mathbb{G}(0)/\widetilde{\Phi}(0)$ independent of $n$.
	Therefore, to complete the proof, it suffices to show that 
	\begin{equation} \label{eq:fin}
		\limsup_{n\to\infty}\frac{\widetilde{\Phi}(X_n)}{\Phi(n)} \leq	\limsup_{n\to\infty} \frac{\Phi((X_n+M)^4)}{\Phi(n)} < \infty \quad \text{ a.s.} 
	\end{equation}
	and that $X$ has at most finitely many cut times almost surely.

	We first apply \cref{thm:import} to prove that $X$ has finitely many cut times almost surely. Let $N\geq3$ be large enough such that $\Phi(N-1)<f(2)$. We can calculate
	\begin{multline*}
		\sum_{n=N}^\infty \frac{1}{D(n)\log n}=\sum_{n=N}^\infty\frac{\widetilde{\Phi}(n)-\widetilde{\Phi}(n+1)}{\widetilde{\Phi}(n)\log n} \\ \leq\sum_{n=N}^\infty\frac{-\widetilde{\Phi}^\prime(n)}{\widetilde{\Phi}(n)\log n}\leq\int_{N-1}^\infty\frac{-\widetilde{\Phi}^\prime(x)}{\widetilde{\Phi}(x)\log x}\ \mathrm{d} x\leq \int_{0}^{\Phi(N-1)} \frac{\mathrm{d} u}{u\log\widetilde{\Phi}^{-1}(u)}<\infty,
	\end{multline*}
	where the first equality follows from \eqref{eq:DGrel}, the first inequality is by convexity, the second follows by integral comparison as  $(\mathrm{d}/\mathrm{d}x)[\log\widetilde{\Phi}(x)]$ is increasing, the third follows by the substitution $u=\widetilde{\Phi}(x)$ and the inequality $\widetilde{\Phi}\leq\Phi$, and the fourth follows from \eqref{eq:stillfinite}. The claim then follows from Theorem \ref{thm:import}.

	Finally, we prove \eqref{eq:fin}. As $\widetilde{\Phi}$ is decreasing, it suffices to show that
	\begin{equation}\label{eq:speedcon}
	 \inf_{m\geq n} X_m \geq n^{1/2-o(1)} \quad \text{a.s.\ as $n\to\infty$ } \qquad \text{ and hence that } \qquad	\liminf_{n\rightarrow\infty} \frac{X_n}{n^{1/4}}>1 \quad \text{a.s.}
	\end{equation}
    Since $\widetilde{\Phi}$ is log-convex, $\widetilde{\Phi}(m+1) \geq \widetilde{\Phi}(m) \widetilde{\Phi}(1)/\widetilde{\Phi}(0)$ for every $m\geq 1$.
	For each $m\geq 0$, define $H_m=\abs{\{n\geq 0:X_n=m\}}$. By \cite[Lemma B]{MR2644879}, for each $m\geq 0$, $H_m$ is a geometric random variable with parameter
	\begin{multline*}1/\mu_m:=\frac{1+2p_m}{2}\cdot \frac{\widetilde{\Phi}(m)-\widetilde{\Phi}(m+1)}{\widetilde{\Phi}(m)}\geq \frac{1}{2}\frac{\widetilde{\Phi}(m)-\widetilde{\Phi}(m+1)}{\widetilde{\Phi}(m)} \\ \geq\frac{-\widetilde{\Phi}^\prime(m+1)}{2\widetilde{\Phi}(m)}\geq -\frac{\widetilde{\Phi}(1)\widetilde{\Phi}^\prime(m+1)}{2\widetilde{\Phi}(0)\widetilde{\Phi}(m+1)}\geq \frac{\widetilde{\Phi}(1)}{4\widetilde{\Phi}(0)(m+2)\sqrt{\log (m+2)}},
	\end{multline*}
	where the second inequality follows by convexity of $\widetilde{\Phi}$ and the final inequality follows from \eqref{eq:loweblog}. Since each $H_m$ is a geometric random variable with mean of order $m^{1+o(1)}$, it follows by an elementary application of Borel-Cantelli that $\max_{m\leq n} H_m = n^{1+o(1)}$ almost surely as $n\to\infty$, and hence that $\max_{m\leq n} X_m \geq n^{1/2-o(1)}$ almost surely as $n\to\infty$. On the other hand, letting $\tau_n$ be the hitting time of $n$ for each $n\geq 1$, we have by optional stopping that
    \[
\mathbb{P}(\text{$X$ visits $m$ after $\tau_n$}) \leq \frac{\mathbb{G}(n)}{\mathbb{G}(m)} =  \frac{\widetilde{\Phi}(n)}{\widetilde{\Phi}(m)} \leq \frac{f(n)}{f(m)}
    \]
    for each $0\leq m \leq n$. Since $f(2^k)/f(\lfloor 2^{(1-\varepsilon)k}\rfloor)$ is superpolynomially small in $k$ for each fixed $\varepsilon>0$, we deduce by a further simple Borel-Cantelli argument that $\inf_{m\geq n} X_m = (\max_{m\leq n} X_m)^{1-o(1)} \geq n^{1/2-o(1)}$ almost surely as $n\to\infty$, completing the proof.
\end{proof}

\subsection*{Acknowledgements}
TH thanks his former advisor Asaf Nachmias for introducing him to the problem in 2013. NH has been supported by the doctoral training centre, Cambridge Mathematics of Information (CMI).

\setstretch{1}
\footnotesize{
	\bibliographystyle{abbrv}
	\bibliography{bibliography}

\begin{thebibliography}{10}

\bibitem{MR3189433}
O.~Angel, N.~Crawford, and G.~Kozma.
\newblock Localization for linearly edge reinforced random walks.
\newblock {\em Duke Math. J.}, 163(5):889--921, 2014.

\bibitem{MR2308594}
I.~Benjamini, O.~Gurel-Gurevich, and R.~Lyons.
\newblock Recurrence of random walk traces.
\newblock {\em Ann. Probab.}, 35(2):732--738, 2007.

\bibitem{MR2789585}
I.~Benjamini, O.~Gurel-Gurevich, and O.~Schramm.
\newblock Cutpoints and resistance of random walk paths.
\newblock {\em Ann. Probab.}, 39(3):1122--1136, 2011.

\bibitem{MR4059006}
I.~Benjamini and J.~Hermon.
\newblock Recurrence of {M}arkov chain traces.
\newblock {\em Ann. Inst. Henri Poincar\'{e} Probab. Stat.}, 56(1):734--759,
  2020.

\bibitem{MR1254826}
I.~Benjamini and Y.~Peres.
\newblock Tree-indexed random walks on groups and first passage percolation.
\newblock {\em Probab. Theory Related Fields}, 98(1):91--112, 1994.

\bibitem{MR1983173}
S.~Blach\`ere.
\newblock Cut times for random walks on the discrete {H}eisenberg group.
\newblock {\em Ann. Inst. H. Poincar\'{e} Probab. Statist.}, 39(4):621--638,
  2003.

\bibitem{MR2408585}
S.~Blach\`ere, P.~Ha\"{\i}ssinsky, and P.~Mathieu.
\newblock Asymptotic entropy and {G}reen speed for random walks on countable
  groups.
\newblock {\em Ann. Probab.}, 36(3):1134--1152, 2008.

\bibitem{MR587211}
F.~T. Bruss.
\newblock A counterpart of the {B}orel-{C}antelli lemma.
\newblock {\em J. Appl. Probab.}, 17(4):1094--1101, 1980.

\bibitem{MR1035664}
K.~Burdzy and G.~F. Lawler.
\newblock Nonintersection exponents for {B}rownian paths. {I}. {E}xistence and
  an invariance principle.
\newblock {\em Probab. Theory Related Fields}, 84(3):393--410, 1990.

\bibitem{MR1062056}
K.~Burdzy and G.~F. Lawler.
\newblock Nonintersection exponents for {B}rownian paths. {II}. {E}stimates and
  applications to a random fractal.
\newblock {\em Ann. Probab.}, 18(3):981--1009, 1990.

\bibitem{MR837740}
T.~K. Carne.
\newblock A transmutation formula for {M}arkov chains.
\newblock {\em Bull. Sci. Math. (2)}, 109(4):399--405, 1985.

\bibitem{MR1118560}
M.~C. Cranston and T.~S. Mountford.
\newblock An extension of a result of {B}urdzy and {L}awler.
\newblock {\em Probab. Theory Related Fields}, 89(4):487--502, 1991.

\bibitem{MR2644879}
E.~Cs\'{a}ki, A.~F\"{o}ldes, and P.~R\'{e}v\'{e}sz.
\newblock On the number of cutpoints of the transient nearest neighbor random
  walk on the line.
\newblock {\em J. Theoret. Probab.}, 23(2):624--638, 2010.

\bibitem{MR556418}
P.~Diaconis and D.~Freedman.
\newblock de {F}inetti's theorem for {M}arkov chains.
\newblock {\em Ann. Probab.}, 8(1):115--130, 1980.

\bibitem{MR947005}
B.~Duplantier.
\newblock Intersections of random walks. {A} direct renormalization approach.
\newblock {\em Comm. Math. Phys.}, 117(2):279--329, 1988.

\bibitem{MR126299}
P.~Erd\H{o}s and S.~J. Taylor.
\newblock Some intersection properties of random walk paths.
\newblock {\em Acta Math. Acad. Sci. Hungar.}, 11:231--248, 1960.

\bibitem{MR3152724}
M.~Folz.
\newblock Volume growth and stochastic completeness of graphs.
\newblock {\em Trans. Amer. Math. Soc.}, 366(4):2089--2119, 2014.

\bibitem{MR2459947}
N.~James, R.~Lyons, and Y.~Peres.
\newblock A transient {M}arkov chain with finitely many cutpoints.
\newblock In {\em Probability and statistics: essays in honor of {D}avid {A}.
  {F}reedman}, volume~2 of {\em Inst. Math. Stat. (IMS) Collect.}, pages
  24--29. Inst. Math. Statist., Beachwood, OH, 2008.

\bibitem{MR1687097}
N.~James and Y.~Peres.
\newblock Cutpoints and exchangeable events for random walks.
\newblock {\em Teor. Veroyatnost. i Primenen.}, 41(4):854--868, 1996.

\bibitem{MR704539}
V.~A. Ka\u{\i}manovich and A.~M. Vershik.
\newblock Random walks on discrete groups: boundary and entropy.
\newblock {\em Ann. Probab.}, 11(3):457--490, 1983.

\bibitem{MR1189088}
G.~F. Lawler.
\newblock Escape probabilities for slowly recurrent sets.
\newblock {\em Probab. Theory Related Fields}, 94(1):91--117, 1992.

\bibitem{MR1423466}
G.~F. Lawler.
\newblock Cut times for simple random walk.
\newblock {\em Electron. J. Probab.}, 1:no. 13, approx. 24 pp., 1996.

\bibitem{MR1386294}
G.~F. Lawler.
\newblock Hausdorff dimension of cut points for {B}rownian motion.
\newblock {\em Electron. J. Probab.}, 1:no. 2, approx. 20 pp., 1996.

\bibitem{MR2985195}
G.~F. Lawler.
\newblock {\em Intersections of random walks}.
\newblock Modern Birkh\"{a}user Classics. Birkh\"{a}user/Springer, New York,
  2013.
\newblock Reprint of the 1996 edition.

\bibitem{MR2009371}
A.~Lejay.
\newblock Simulating a diffusion on a graph. {A}pplication to reservoir
  engineering.
\newblock {\em Monte Carlo Methods Appl.}, 9(3):241--255, 2003.

\bibitem{LevyPaul1954Tdld}
P.~L{\'e}vy.
\newblock {\em Th{\'e}orie de l'addition des variables al{\'e}atoires}.
\newblock Monographies des probabilit{\'e}s ; calcul des probabilit{\'e}s et
  ses applications. Gauthier-Villars, Paris, 2. ed. edition, 1954.

\bibitem{dur34690}
C.~H. Lo, M.~V. Menshikov, and A.~R. Wade.
\newblock Cutpoints of non-homogeneous random walks.
\newblock {\em ALEA - Latin American Journal of Probability and Mathematical
  Statistics}, 2021.

\bibitem{MR3502602}
T.~Lupu.
\newblock From loop clusters and random interlacements to the free field.
\newblock {\em Ann. Probab.}, 44(3):2117--2146, 2016.

\bibitem{MR3616205}
R.~Lyons and Y.~Peres.
\newblock {\em Probability on Trees and Networks}, volume~42 of {\em Cambridge
  Series in Statistical and Probabilistic Mathematics}.
\newblock Cambridge University Press, New York, 2016.
\newblock Available at \url{https://rdlyons.pages.iu.edu/}.

\bibitem{MR4228277}
R.~Lyons and Y.~Peres.
\newblock Poisson boundaries of lamplighter groups: proof of the
  {K}aimanovich-{V}ershik conjecture.
\newblock {\em J. Eur. Math. Soc. (JEMS)}, 23(4):1133--1160, 2021.

\bibitem{MR2441859}
F.~Merkl, A.~\"{O}ry, and S.~W.~W. Rolles.
\newblock The `magic formula' for linearly edge-reinforced random walks.
\newblock {\em Statist. Neerlandica}, 62(3):345--363, 2008.

\bibitem{MR822826}
N.~T. Varopoulos.
\newblock Long range estimates for {M}arkov chains.
\newblock {\em Bull. Sci. Math. (2)}, 109(3):225--252, 1985.

\end{thebibliography}
}

\end{document}